\documentclass[12pt,reqno]{amsart}
\usepackage[margin=1in]{geometry}
\usepackage{amsmath,amssymb,amsthm,graphicx,amsxtra, setspace}
\usepackage[utf8]{inputenc}
\usepackage{mathrsfs}
\usepackage{hyperref}
\usepackage{upgreek}
\usepackage{mathtools}
\allowdisplaybreaks

\usepackage[pagewise]{lineno}

\newtheorem{theorem}{Theorem}[section]
\newtheorem{lemma}[theorem]{Lemma}
\newtheorem{proposition}[theorem]{Proposition}
\newtheorem{assumption}[theorem]{Assumption}
\newtheorem{corollary}[theorem]{Corollary}
\newtheorem{definition}[theorem]{Definition}

\newtheorem{remark}[theorem]{Remark}

\let\originalleft\left
\let\originalright\right
\renewcommand{\left}{\mathopen{}\mathclose\bgroup\originalleft}
\renewcommand{\right}{\aftergroup\egroup\originalright}

\renewcommand{\d}{\/\mathrm{d}\/}

\def\w{\textbf{W}^{\varepsilon}_{{\theta}^{\varepsilon}}}

\def\e{\varepsilon}

\def\S{\mathrm{S}}

\def\L{\mathbb{L}}
\def\A{\mathrm{A}}

\def\C{\mathrm{C}}
\def\f{\mathbf{f}}

\def\B{\mathrm{B}}
\def\D{\mathrm{D}}
\def\y{\mathbf{y}}

\def\K{\mathrm{K}}
\def\E{\mathbb{E}}

\def\x{\mathbf{x}}
\def\k{\mathbf{k}}

\def\z{\mathbf{z}}
\def\v{\mathbf{v}}
\def\V{\mathbb{v}}
\def\w{\mathbf{w}}
\def\W{\mathrm{W}}
\def\G{\mathrm{G}}

\def\M{\mathrm{M}}
\def\N{\mathbb{N}}

\def\V{\mathbb{V}}
\def\wi{\widetilde}

\def\u{\mathrm{U}}

\def\u{\mathbf{u}}
\def\H{\mathbb{H}}

\newcommand{\R}{\mathbb{R}}

\renewcommand{\d}{\/\mathrm{d}\/}

\newcommand{\Addresses}{{% additional braces for segregating \footnotesize
		\footnote{
			\noindent \textsuperscript{1,2}Department of Mathematics, Indian Institute of Technology Roorkee-IIT Roorkee,
			Haridwar Highway, Roorkee, Uttarakhand 247667, INDIA.\par\nopagebreak
			\noindent  \textit{e-mail:} \texttt{Manil T. Mohan: maniltmohan@ma.iitr.ac.in, maniltmohan@gmail.com.}
			
			\textit{e-mail:} \texttt{Kush Kinra: kkinra@ma.iitr.ac.in.}
			
			\noindent \textsuperscript{*}Corresponding author.
			
			\textit{Key words:} 3D stochastic convective Brinkman-Forchheimer equations, periodic domains, cylindrical Wiener process, random dynamical system, absorbing sets, random attractors  
			
			Mathematics Subject Classification (2020): Primary 35B41, 35Q35; Secondary 37L55, 37N10, 35R60.

}}}

\begin{document}
	%	\linenumbers
	\title[Large time behavior of 3D deterministic and stochastic CBF equations]{Large time behavior of the deterministic and stochastic 3D convective Brinkman-Forchheimer equations in periodic domains
		\Addresses}
	\author[K. Kinra and M. T. Mohan]
	{Kush Kinra\textsuperscript{1} and Manil T. Mohan\textsuperscript{2*}}

	\maketitle
	\begin{abstract}
		The large time behavior of the deterministic and stochastic three dimensional convective Brinkman-Forchheimer (CBF) equations $$\partial_t\u-\mu \Delta\u+(\u\cdot\nabla)\u+\alpha\u+\beta|\u|^{r-1}\u+\nabla p=\mathbf{f},\ \nabla\cdot\u=0,$$ for $r\geq3$ ($r>3$, for any $\mu$ and $\beta$, and $r=3$ for $2\beta\mu\geq1$), in periodic domains is carried out in this work. Our first goal is to prove the existence of global attractors for the 3D deterministic CBF equations. Then, we show the existence of random attractors for the 3D stochastic CBF equations perturbed by small additive smooth noise. Finally, we  establish the upper semicontinuity of random attractor for the 3D stochastic CBF equations (stability of attractors), when the coefficient of random perturbation approaches to zero.
	\end{abstract}
	\section{Introduction} \label{sec1}\setcounter{equation}{0}
	The analysis of long time behavior of the deterministic evolution equations, especially Navier-Stokes equations (NSE) is well-investigated in \cite{Robinson1,R.Temam}, etc. Attractors are one of the most important entities in the study of long time behavior of dynamical systems generated by dissipative evolution equations. When partial differential equations (PDEs)  are perturbed by random  noises, which generate  random dynamical systems (RDS),  it is important to examine  the random attractors for such systems and their stability, that is, the convergence of the random attractors  (cf. \cite{Arnold}).  The analysis of long time behavior of infinite dimensional RDS is also an essential branch in the study of qualitative properties of stochastic PDEs (see \cite{BCF,Crauel1,CF}, etc for more details).	In the literature, a variety of results on the global and random attractors for several deterministic and stochastic models are available, the interested readers are referred to see \cite{ATT,BCL,CFT,FMRT,KT,KZ,OAL1,MTM4,Sell} etc for global attractors   and see \cite{HCPEK,FY,FS,GS,GGW,GLW,KM1,FLYB,LG,YFWLM,You,You1,ZD} etc,  for random attractors, and the references therein.
	
	This work is concerned about the long time behavior of three dimensional deterministic and stochastic convective Brinkman-Forchheimer equations in periodic domains. Let $\mathbb{T}^3 \subset \R^3$ be a periodic domain.  Let  $\u(x,t) :
	\mathbb{T}^3\times(0,\infty)\to \R^3$ denote the velocity field and  $p(x,t):
	\mathbb{T}^3\times(0,\infty)\to\R$ represent the pressure field. We consider the convective Brinkman-Forchheimer (CBF) equations, which describe the motion of incompressible fluid flows in a saturated porous medium, as:
	\begin{equation}\label{1}
	\left\{
	\begin{aligned}
	\frac{\partial \u}{\partial t}-\mu \Delta\u+(\u\cdot\nabla)\u+\alpha\u+\beta|\u|^{r-1}\u+\nabla p&=\mathbf{f}, \ \text{ in } \ \mathbb{T}^3\times(0,\infty), \\ \nabla\cdot\u&=0, \ \text{ in } \ \mathbb{T}^3\times(0,\infty), \\
	\u(0)&=\x, \ \text{ in } \ \mathbb{T}^3,
	\end{aligned}
	\right.
	\end{equation}
	where $\f(x,t):
	\mathbb{T}^3\times(0,\infty)\to \R^3$ is an external forcing. The positive constants $\mu,\alpha$ and $\beta$ represent  the Brinkman coefficient (effective viscosity), the Darcy coefficient (permeability of porous medium) and Forchheimer coefficient, respectively. The exponent $r\in[1,\infty)$ is called the absorption exponent and  $r=3$ is known as the critical exponent. For $\alpha=\beta=0$, we obtain the classical 3D Navier-Stokes equations (NSE). The applicability of CBF equations \eqref{1} is limited to flows when the velocities are sufficiently high and porosities are not too small, that is, when the Darcy law for a porous medium no longer applies (cf. \cite{PAM}). It should be noted that the critical homogeneous CBF equations \eqref{1} have the same scaling as the NSE only when $\alpha=0$ (see Proposition 1.1, \cite{HR} and no scale invariance property for other values of $\alpha$ and $r$), which is sometimes referred to as the NSE modified by an absorption term (\cite{SNA}) or the tamed NSE (\cite{MRXZ}). The existence and uniqueness of Leray-Hopf weak solutions satisfying the energy equality as well as strong solutions (in the deterministic sense)  for  the 3D CBF equations \eqref{1} for $r\geq 3$ ($2\beta\mu\geq 1,$ for $r=3$) is established in  \cite{FHR,HR,PAM,MTM}, etc. Likewise the 3D NSE, the global existence and uniqueness of strong solutions for the 3D SCBF equations \eqref{1} with $r\in[1,3)$ is still an open problem. Therefore, in this work, we consider the cases $r>3$, for any $\mu$ and $\beta$, and  $r=3,$  for $2\beta\mu\geq1$ only.

	In the first part of this work, we prove the existence of a global attractor for the 3D SCBF equations \eqref{1}. The following estimate plays a key role in obtaining such a result  (cf. \cite{HR}). 
	\begin{align}\label{3}
	&	\int_{\mathbb{T}^3}(-\Delta\u(x))\cdot|\u(x)|^{r-1}\u(x)d x\nonumber\\&=\int_{\mathbb{T}^3}|\nabla\u(x)|^2|\u(x)|^{r-1}\d x+4\left[\frac{r-1}{(r+1)^2}\right]\int_{\mathbb{T}^3}|\nabla|\u(x)|^{\frac{r+1}{2}}|^2\d x.
	\end{align}
	Note that in the case of bounded domains  $\mathcal{P}(|\u|^{r-1}\u)$ ($\mathcal{P}$ is the Helmholtz-Hodge projection) need not be zero on the boundary, and $\mathcal{P}$ and $-\Delta$ are not necessarily commuting (for a counter example, see Example 2.19, \cite{RRS}). Thus the equality \eqref{3} may not be useful in the context of bounded domains and we restrict ourselves to periodic domains in this work. 	In the case of two dimensional bounded and unbounded domains, the existence of global attractors and their properties for the  CBF equations \eqref{1} have been discussed in \cite{MTM2,MTM3}, etc. 
	
	In the second part of the work, we consider the stochastic convective Brinkman-Forchheimer  (SCBF)  equations and discuss about the long time behavior of its solution and stability results. The 3D SCBF equations in periodic domains are given by 
	\begin{equation}\label{2}
	\left\{
	\begin{aligned}
	\d\u+[-\mu \Delta\u+(\u\cdot\nabla)\u+\alpha\u+\beta|\u|^{r-1}\u+\nabla p]\d t&=\mathbf{f}\d t+\e\d\W(t), \ \text{ in } \ \mathbb{T}^3\times(0,\infty), \\ \nabla\cdot\u&=0, \ \text{ in } \ \mathbb{T}^3\times(0,\infty), \\
	\u(0)&=\x, \ \text{ in } \ \mathbb{T}^3,
	\end{aligned}
	\right.
	\end{equation}
	where	$\text W(t), \ t\in \R,$ is a two-sided cylindrical Wiener process  with its Reproducing Kernel Hilbert Space (RKHS)  $\K$ satisfying certain assumptions (see subsection \ref{sub4} below).   Using monotonicity as well as hemicontinuity properties of linear and nonlinear operators, and a stochastic generalization of the Minty-Browder technique, the existence and uniqueness of pathwise strong solutions (in the probabilistic sense) of the system \eqref{2} is established in \cite{MTM1}. The random attractors and their stability results for the 2D SCBF equations in bounded and unbounded domains like Poincar\'e domains is discussed in \cite{KM, KM1,KM2}, etc. 	The existence of a random attractor in $\V$ for the 3D damped Navier-Stokes equations perturbed by additive noise  with $3<r\leq 5$ in bounded domains is established in \cite{You} by using the equality \eqref{3}. Due to the technical difficulty discussed earlier, it appears to us that the results obtained in \cite{You} are only valid in periodic domains.

	Firstly, we establish the existence of a random attractor in $\H$ for the system \eqref{2} for $r\geq 3$ ($2\beta\mu\geq 1,$ for $r=3$). Then, we prove one important property of the random attractors, namely \emph{the upper semicontinuity} of random  attractors, which was introduced in \cite{CLR}.  Roughly speaking, the upper semicontinuity of random  attractors means that, if $\mathcal{A}$ is a global attractor for the deterministic system and $\mathcal{A}_\varepsilon$ is a random attractor for the corresponding stochastic system perturbed by a small noise, we say that these attractors have the property of upper semicontinuity if $$\lim\limits_{\varepsilon\to 0}d(\mathcal{A}_\varepsilon, \mathcal{A})=0,$$ where $d$ is the Hausdorff semidistance given by $d(A, B)=\sup\limits_{y\in A}\inf\limits_{z\in B}\rho(y,z),$ for any $A, B\subset X,$ on a Polish space $(X,\rho)$. The upper semicontinuity of random attractors for the 2D SNSE and stochastic reaction-diffusion equations is established in \cite{CLR}. The existence and  upper semicontinuity of random attractors for several stochastic models are available in \cite{XJXD,KM1,LCL,WLYJ}, etc. In particular, upper semicontinuity results for the 2D SCBF equations in bounded domains is discussed in \cite{KM1}. We prove the upper semicontinuity of random attractors for the 3D SCBF equations \eqref{2} by using the similar techniques given in \cite{CLR}.
	
	The organization of the paper is as follows. In section \ref{sec2}, we provide the functional framework needed for this work. Also, we define linear and nonlinear operators, and discuss about their properties. After providing an abstract formulation of the system \eqref{1}, we discuss about the global solvability results of the system \eqref{1} in the same section. The section \ref{sec3} is devoted for establishing the existence of global attractors for the deterministic CBF equations \eqref{1} with  $r\geq 3$ ($2\beta\mu\geq 1,$ for $r=3$) (Theorem \ref{d-ab}). In section \ref{sec4}, we first give the abstract formulation of the 3D SCBF equations \eqref{2}, the assumption on RKHS $K$ and the existence and uniqueness results of the system \eqref{2}. Then, we define the metric dynamical system (MDS) and random dynamical system (RDS) for the 3D SCBF equations \eqref{1}. We prove the existence of random attractors for the  3D SCBF equations \eqref{2} in the same section (Theorem \ref{Main_Theoem}). The upper semicontinuity of random attractors for the 3D SCBF equations \eqref{2} using the theory given in \cite{CLR} is discussed in the final section (Theorem \ref{USC}).

	\section{Mathematical Formulation}\label{sec2}\setcounter{equation}{0}
	In this section, we provide the necessary function spaces needed to obtain the existence of global attractors and random attractors for the equations \eqref{1} and \eqref{2}, respectively. We consider the problem \eqref{1} on a three dimensional torus $\mathbb{T}^3=[0,L]^3$,  with the periodic boundary conditions and zero-mean value constraint for the functions, that is, $\int_{\mathbb{T}^3}\u(x)\d x=0$. Since, $\alpha$ does not play a major role in our analysis, we will put $\alpha=0$ in \eqref{1} and \eqref{2} for further analysis.
	\subsection{Function spaces} Let \ $\dot{\C}_p^{\infty}(\mathbb{T}^3;\R^3)$ denote the space of all infinitely\ differentiable  functions ($\mathbb{R}^3$-valued) such that $\int_{\mathbb{T}^3}\u(x)\d x=0$ and $\u(x+\mathrm{L}e_i)=\u(x),$ for
	every $x\in\R^3$ and $i=1,2,3$, where
	$\{e_1,e_2,e_3\}$ is the canonical basis
	of $\R^3$. The Sobolev space  $\dot{\H}_p^k(\mathbb{T}^3):=\dot{\mathrm{H}}_p^k(\mathbb{T}^3;\mathbb{R}^3)$ is the completion of $\dot{\C}_p^{\infty}(\mathbb{T}^3;\R^3)$  with respect to the $\H^s$ norm $$\|\u\|_{\dot{\H}^s_p}:=\left(\sum_{0\leq|\alpha|\leq s}\|\D^{\alpha}\u\|_{\mathbb{L}^2(\mathcal{O})}^2\right)^{1/2}.$$ The Sobolev space of periodic functions with zero mean $\dot{\H}_p^k(\mathbb{T}^3)$ is the same as (Proposition 5.39, \cite{Robinson1}) $$\left\{\u:\u=\sum_{\k\in\mathbb{Z}^3}\u_\k e^{2\pi \k\cdot\x /  L},\u_0=\mathbf{0},\ \bar{\u}_\k=\u_{-\k},\ \|\u\|_{\dot{\H}^s_f}:=\sum_{k\in\mathbb{Z}^m}|\k|^{2s}|\u_\k|^2<\infty\right\}.$$ From Proposition 5.38, \cite{Robinson1}, we infer that the norms $\|\cdot\|_{\dot{\H}^s_p}$ and $\|\cdot\|_{\dot{\H}^s_f}$ are equivalent. Let us define 
	\begin{align*} 
	\mathcal{V}&:=\{\u\in\dot{\C}_p^{\infty}(\mathbb{T}^3;\R^3):\nabla\cdot\u=0\},\\
	\mathbb{H}&:=\text{the closure of }\ \mathcal{V} \ \text{ in the Lebesgue space } \L^2(\mathbb{T}^3)=\mathrm{L}^2(\mathbb{T}^3;\R^3),\\
	\mathbb{V}&:=\text{the closure of }\ \mathcal{V} \ \text{ in the Sobolev space } \H^1(\mathbb{T}^3)=\mathrm{H}^1(\mathbb{T}^3;\R^3),\\
	\widetilde{\L}^{p}&:=\text{the closure of }\ \mathcal{V} \ \text{ in the Lebesgue space } \L^p(\mathbb{T}^3)=\mathrm{L}^p(\mathbb{T}^3;\R^3),
	\end{align*}
	for $p\in(2,\infty)$. The zero mean condition provides the \emph{Poincar\'{e} inequality}, \begin{align}\label{poin}
	\lambda_1\|\u\|_{\mathbb{H}}^2\leq\|\u\|^2_{\V},
	\end{align} where $\lambda_1=\frac{2\pi}{L}$ (Lemma 5.40, \cite{Robinson1}). Then, we characterize the spaces $\H$, $\V$ and $\widetilde{\L}^p$   with the norms  $$\|\u\|_{\H}^2:=\int_{\mathbb{T}^3}|\u(x)|^2\d x,\quad \|\u\|_{\V}^2:=\int_{\mathbb{T}^3}|\nabla\u(x)|^2\d x \ \text{ and }\ \|\u\|_{\widetilde{\L}^p}^p=\int_{\mathbb{T}^3}|\u(x)|^p\d x,$$ respectively. 
	Let $(\cdot,\cdot)$ denote the inner product in the Hilbert space $\H$ and $\langle \cdot,\cdot\rangle $ represent the induced duality between the spaces $\V$  and its dual $\V'$ as well as $\widetilde{\L}^p$ and its dual $\widetilde{\L}^{p'}$, where $\frac{1}{p}+\frac{1}{p'}=1$. Note that $\H$ can be identified with its dual $\H'$. We endow the space $\V\cap\widetilde{\L}^{p}$ with the norm $\|\u\|_{\V}+\|\u\|_{\widetilde{\L}^{p}},$ for $\u\in\V\cap\widetilde{\L}^p$ and its dual $\V'+\widetilde{\L}^{p'}$ with the norm $$\inf\left\{\max\left(\|\v_1\|_{\V'},\|\v_1\|_{\widetilde{\L}^{p'}}\right):\v=\v_1+\v_2, \ \v_1\in\V', \ \v_2\in\widetilde{\L}^{p'}\right\}.$$ Moreover, we have the continuous embedding $\V\cap\widetilde{\L}^p\hookrightarrow\H\hookrightarrow\V'+\widetilde{\L}^{p'}$. We use the following interpolation inequality in the sequel. 
	Assume $1\leq s_1\leq s\leq s_2\leq \infty$, $\theta\in(0,1)$ such that $\frac{1}{s}=\frac{\theta}{s_1}+\frac{1-\theta}{s_2}$ and $\u\in\L^{s_1}(\mathbb{T}^3)\cap\L^{s_2}(\mathbb{T}^3)$, then we have 
	\begin{align}\label{211}
	\|\u\|_{\L^s}\leq\|\u\|_{\L^{s_1}}^{\theta}\|\u\|_{\L^{s_2}}^{1-\theta}. 
	\end{align}
	
	\subsection{Linear operator}
	Let $\mathcal{P}: \dot{\L}^2(\mathbb{T}^3) \to\H$ denote the Helmholtz-Hodge (or Leray) projection (section 2.1, \cite{RRS}).  We define the Stokes operator 
	\begin{equation*}
	\A\u:=-\mathcal{P}\Delta\u,\;\u\in\D(\A):=\V\cap\dot{\H}^{2}_p(\mathbb{T}^3).
	\end{equation*}
	Note that $\D(\A)$ can also be written as $\D(\A)=\big\{\u\in\dot{\H}^{2}_p(\mathbb{T}^3):\nabla\cdot\u=0\big\}$.  It should be noted that $\mathcal{P}$ and $\Delta$ commutes in periodic domains (Lemma 2.9, \cite{RRS}). For the Fourier expansion $\u=\sum\limits_{\k\in\mathbb{Z}^3} e^{2\pi \k\cdot\x /  L}\u_\k,$ one obtains $$-\Delta\u=\frac{4\pi^2}{L^2}\sum_{\k\in\mathbb{Z}^3} e^{2\pi \k\cdot\x /  L}|\k|^2\u_\k.$$ It is easy to observe that $\D(\A^{s/2})=\big\{\u\in \dot{\H}^{s}_p(\mathbb{T}^3):\nabla\cdot\u=0\big\}$ and $\|\A^{s/2}\u\|_{\H}=C\|\u\|_{\dot{\H}^{s}_p},$ for all $\u\in\D(\A^{s/2})$, $s\geq 0$ (cf. \cite{Robinson1}). Note that the operator $\A$ is a non-negative self-adjoint operator in $\H$ with a compact resolvent and   \begin{align}\label{2.7a}\langle \A\u,\u\rangle =\|\u\|_{\V}^2,\ \textrm{ for all }\ \u\in\V, \ \text{ so that }\ \|\A\u\|_{\V'}\leq \|\u\|_{\V}.\end{align}
	Since $\A^{-1}$ is a compact self-adjoint operator in $\H$, we obtain a complete family of orthonormal eigenfunctions  $\{w_i\}_{i=1}^{\infty}\subset\dot{\C}_p^{\infty}(\mathbb{T}^3;\R^3)$ such that $\A w_i=\lambda_i w_i$, for $i=1,2,\ldots$ and  $0<\lambda_1\leq \lambda_2\leq \ldots\to\infty$ are the eigenvalues of $\A$. Note that $\lambda_1=\frac{2\pi}{L}$ is the smallest eigenvalue of $\A$ appearing in the Poincar\'e inequality \eqref{poin}.

	\subsection{Bilinear operator}
	Let us define the \emph{trilinear form} $b(\cdot,\cdot,\cdot):\V\times\V\times\V\to\R$ by $$b(\u,\v,\w)=\int_{\mathbb{T}^3}(\u(x)\cdot\nabla)\v(x)\cdot\w(x)\d x=\sum_{i,j=1}^3\int_{\mathbb{T}^3}\u_i(x)\frac{\partial \v_j(x)}{\partial x_i}\w_j(x)\d x.$$ If $\u, \v$ are such that the linear map $b(\u, \v, \cdot) $ is continuous on $\V$, the corresponding element of $\V'$ is denoted by $\B(\u, \v)$. We also denote $\B(\u) = \B(\u, \u)=\mathcal{P}(\u\cdot\nabla)\u$.
	An integration by parts gives 
	\begin{equation}\label{b0}
	\left\{
	\begin{aligned}
	b(\u,\v,\w) &=  -b(\u,\w,\v),\ \text{ for all }\ \u,\v,\w\in \V,\\
	b(\u,\v,\v) &= 0,\ \text{ for all }\ \u,\v \in\V.
	\end{aligned}
	\right.\end{equation}
	\begin{remark}
		The following well-known inequality is due to Ladyzhenskaya (Lemma 2, Chapter I, \cite{OAL}):
		\begin{align}\label{lady}
		\|\u\|_{\L^{4} (\mathbb{T}^3)} \leq 2^{1/2} \|\u\|^{1/4}_{\L^{2} (\mathbb{T}^3)} \|\nabla \u\|^{3/4}_{\L^{2} (\mathbb{T}^3)}, \ \ \ \u\in \H^{1} (\mathbb{T}^3).
		\end{align}
	\end{remark}
	\begin{remark}
		In the trilinear form, an application of H\"older's inequality yields
		\begin{align*}
		|b(\u_1,\u_2,\u_3)|=|b(\u_1,\u_3,\u_2)|\leq \|\u_1\|_{\widetilde{\L}^{r+1}}\|\u_2\|_{\widetilde{\L}^{\frac{2(r+1)}{r-1}}}\|\u_3\|_{\V},
		\end{align*}
		for all $\u_1\in\V\cap\widetilde{\L}^{r+1}$, $\u_2\in\V\cap\widetilde{\L}^{\frac{2(r+1)}{r-1}}$ and $\u_3\in\V$, so that we get 
		\begin{align}\label{2p9}
		\|\B(\u_1,\u_2)\|_{\V'}\leq \|\u_1\|_{\widetilde{\L}^{r+1}}\|\u_2\|_{\widetilde{\L}^{\frac{2(r+1)}{r-1}}}.
		\end{align}
		Hence, the trilinear map $b : \V\times\V\times\V\to \R$ has a unique extension to a bounded trilinear map from $(\V\cap\widetilde{\L}^{r+1})\times(\V\cap\widetilde{\L}^{\frac{2(r+1)}{r-1}})\times\V$ to $\R$. It can also be seen that $\B$ maps $ \V\cap\widetilde{\L}^{r+1}$  into $\V'+\widetilde{\L}^{\frac{r+1}{r}}$ and using interpolation inequality (see \eqref{211}), we get 
		\begin{align}\label{212}
		\left|\langle \B(\u_1),\u_2\rangle \right|=\left|b(\u_1,\u_2,\u_1)\right|\leq \|\u_1\|_{\widetilde{\L}^{r+1}}\|\u_1\|_{\widetilde{\L}^{\frac{2(r+1)}{r-1}}}\|\u_2\|_{\V}\leq\|\u_1\|_{\widetilde{\L}^{r+1}}^{\frac{r+1}{r-1}}\|\u_1\|_{\H}^{\frac{r-3}{r-1}}\|\u_2\|_{\V},
		\end{align}
		for $r\geq 3$ and all $\u_2\in\V\cap\widetilde{\L}^{r+1}$. Thus, we have 
		\begin{align}\label{2.9a}
		\|\B(\u_1)\|_{\V'+\widetilde{\L}^{\frac{r+1}{r}}}\leq\|\u_1\|_{\widetilde{\L}^{r+1}}^{\frac{r+1}{r-1}}\|\u_1\|_{\H}^{\frac{r-3}{r-1}}.
		\end{align}
	\end{remark}
	\begin{remark}
		Note that $\langle\B(\u_1,\u_1-\u_2),\u_1-\u_2\rangle=0$ and it implies that
		\begin{equation}\label{441}
		\begin{aligned}
		\langle \B(\u_1)-\B(\u_2),\u_1-\u_2\rangle =\langle\B(\u_1-\u_2,\u_2),\u_1-\u_2\rangle=-\langle\B(\u_1-\u_2,\u_1-\u_2),\u_2\rangle.
		\end{aligned}
		\end{equation} 
		Using H\"older's and Young's inequalities, we estimate $|\langle\B(\u_1-\u_2,\u_1-\u_2),\u_2\rangle|$ as  
		\begin{align}\label{2p28}
		|\langle\B(\u_1-\u_2,\u_1-\u_2),\u_2\rangle|&\leq\|\u_1-\u_2\|_{\V}\|\u_2(\u_1-\u_2)\|_{\H}\nonumber\\&\leq\frac{\mu }{4}\|\u_1-\u_2\|_{\V}^2+\frac{1}{\mu }\|\u_2(\u_1-\u_2)\|_{\H}^2.
		\end{align}
		We take the term $\|\u_2(\u_1-\u_2)\|_{\H}^2$ from \eqref{2p28} and use H\"older's and Young's inequalities to estimate it as (see \cite{MTM1} also)
		\begin{align}\label{2.29}
		&\int_{\mathbb{T}^3}|\u_2(x)|^2|\u_1(x)-\u_2(x)|^2\d x\nonumber\\&=\int_{\mathbb{T}^3}|\u_2(x)|^2|\u_1(x)-\u_2(x)|^{\frac{4}{r-1}}|\u_1(x)-\u_2(x)|^{\frac{2(r-3)}{r-1}}\d x\nonumber\\&\leq\left(\int_{\mathbb{T}^3}|\u_2(x)|^{r-1}|\u_1(x)-\u_2(x)|^2\d x\right)^{\frac{2}{r-1}}\left(\int_{\mathbb{T}^3}|\u_1(x)-\u_2(x)|^2\d x\right)^{\frac{r-3}{r-1}}\nonumber\\&\leq{\frac{\beta\mu}{8} }\left(\int_{\mathbb{T}^3}|\u_2(x)|^{r-1}|\u_1(x)-\u_2(x)|^2\d x\right)+\frac{r-3}{r-1}\left(\frac{16}{\beta\mu (r-1)}\right)^{\frac{2}{r-3}}\left(\int_{\mathbb{T}^3}|\u_1(x)-\u_2(x)|^2\d x\right),
		\end{align}
		for $r>3$. Using \eqref{2.29} in \eqref{2p28}, we find 
		\begin{align}\label{2.16}
		&|\langle\B(\u_1-\u_2,\u_1-\u_2),\u_2\rangle|\leq\frac{\mu }{4}\|\u_1-\u_2\|_{\V}^2+\frac{\beta}{8}\||\u_2|^{\frac{r-1}{2}}(\u_1-\u_2)\|_{\H}^2+\eta_1\|\u_1-\u_2\|_{\H}^2,
		\end{align}
		where \begin{align}\label{215}\eta_1=\frac{r-3}{\mu(r-1)}\left(\frac{16}{\beta\mu (r-1)}\right)^{\frac{2}{r-3}}.\end{align}
	\end{remark}
	\subsection{Nonlinear operator}
	Let us now consider the operator $\mathcal{C}(\u):=\mathcal{P}(|\u|^{r-1}\u)$. It is immediate that $\langle\mathcal{C}(\u),\u\rangle =\|\u\|_{\widetilde{\L}^{r+1}}^{r+1}$ and the map $\mathcal{C}(\cdot):\widetilde{\L}^{r+1}\to\widetilde{\L}^{\frac{r+1}{r}}$. For all $\u\in\wi\L^{r+1}$, the map is Gateaux differentiable with Gateaux derivative 
	\begin{align}\label{29}
	\mathcal{C}'(\u)\v&=\left\{\begin{array}{cc}\mathcal{P}(\v),&\text{ for }r=1,\\ \left\{\begin{array}{cc}\mathcal{P}(|\u|\v)+\mathcal{P}\left(\frac{\u}{|\u|}(\u\cdot\v)\right),&\text{ if }\u\neq \mathbf{0},\\\mathbf{0},&\text{ if }\u=\mathbf{0},\end{array}\right.&\text{ for } r=2,\\ \mathcal{P}(|\u|^{r-1}\v)+(r-1)\mathcal{P}(\u|\u|^{r-3}(\u\cdot\v)), &\text{ for }r\geq 3,\end{array}\right.
	\end{align}
	for all $\v\in\widetilde{\L}^{r+1}$. For any $r\in [1, \infty)$ and $\u_1, \u_2 \in \V\cap\widetilde{\L}^{r+1}$, we have (see subsection 2.4, \cite{MTM1}),
	\begin{align}\label{MO_c}
	\langle\mathcal{C}(\u_1)-\mathcal{C}(\u_2),\u_1-\u_2\rangle \geq\frac{1}{2}\||\u_1|^{\frac{r-1}{2}}(\u_1-\u_2)\|_{\H}^2+\frac{1}{2}\||\u_2|^{\frac{r-1}{2}}(\u_1-\u_2)\|_{\H}^2\geq 0,
	\end{align}
	and 
	\begin{align}\label{a215}
	\|\u-\v\|_{\wi\L^{r+1}}^{r+1}\leq 2^{r-2}\||\u|^{\frac{r-1}{2}}(\u-\v)\|_{\H}^2+2^{r-2}\||\v|^{\frac{r-1}{2}}(\u-\v)\|_{\H}^2,
	\end{align}
	for $r\geq 1$ (replace $2^{r-2}$ with $1,$ for $1\leq r\leq 2$).
	\begin{theorem}[Theorem 2.2, \cite{MTM1}]\label{thm2.4}
		Let $ r> 3$ and $\u_1, \u_2 \in \V\cap\widetilde{\L}^{r+1}.$ Then, for the operator $\G(\u) = \mu\A\u +\B(\u)+\beta\mathcal{C}(\u),$ we have
		\begin{align}\label{fe4}
		\langle\G(\u_1)-\G(\u_2),\u_1-\u_2\rangle+ \eta_2\|\u_2-\u_2\|_{\H}^2&\geq 0,
		\end{align}
		where \begin{align}\label{216}\eta_2=\frac{r-3}{2\mu(r-1)}\left(\frac{2}{\beta\mu (r-1)}\right)^{\frac{2}{r-3}}.\end{align}
	\end{theorem}
	\begin{theorem}[Theorem 2.3, \cite{MTM1}]\label{thm2.3}
		For $r=3$ with $2\beta\mu \geq 1$, the operator $\G(\cdot):\V\cap\widetilde{\L}^{r+1}\to \V'+\widetilde{\L}^{\frac{r+1}{r}}$ is globally monotone, that is, for all $\u_1,\u_2\in\V$, we have 
		\begin{align}\label{fe5}\langle\G(\u_1)-\G(\u_2),\u_1-\u_2\rangle\geq 0.\end{align}
	\end{theorem}
	\section{Global attractors for 3D deterministic CBF equations}\label{sec3}\setcounter{equation}{0}
	In this section, we obtain the existence of global attractors for the 3D deterministic CBF equations \eqref{1} in periodic domains. We make use of the equality \eqref{3} to obtain the existence of an absorbing set in $\V$. 
	\subsection{Deterministic CBF equations}
	In this subsection, we provide the abstract formulation of deterministic CBF equation \eqref{1} and discuss about the existence and uniqueness of weak  solutions. 
	\subsubsection{Abstract formulation}
	On taking orthogonal projection $\mathcal{P}$ onto the first equation in \eqref{1}, we obtain 
	\begin{equation}\label{D-CBF}
	\left\{
	\begin{aligned}
	\frac{\d\u(t)}{\d t}+\mu \A\u(t)+\B(\u(t))+\beta \mathcal{C}(\u(t))&=\f , \ \ \ t\geq 0, \\ 
	\u(0)&=\x,
	\end{aligned}
	\right.
	\end{equation}
	where $\x\in \H \text{ and } \f\in \H$. Now, we shall give the definition of weak solution of the system \eqref{D-CBF} for $r\geq3$ and discuss about global solvability results. 
	\begin{definition}\label{def3.1}
		Let us assume that $\x\in \H$ and $\f\in \H$. Let $T>0$ be any fixed time. Then, the function $\u(\cdot)$ is called a \emph{Leray-Hopf weak solution} of the problem \eqref{D-CBF} on time interval $[0, T]$, if $$\u\in \mathrm{L}^{\infty}(0,T;\H)\cap\mathrm{L}^2(0, T;\V)\cap\mathrm{L}^{r+1}(0,T;\widetilde{\L}^{r+1}),$$ with $ \partial_t\u\in \mathrm{L}^{2}(0,T;\V')+\mathrm{L}^{\frac{r+1}{r}}(0,T;\widetilde{\L}^{\frac{r+1}{r}})$ satisfying: 
		\begin{enumerate}
			\item [(i)]  	for any $\psi\in \V\cap\widetilde{\L}^{r+1},$ 	\begin{align}
			\bigg\langle\frac{\d\u(t)}{\d t}, \psi\bigg\rangle =  - \big\langle \mu \A\u(t)+\B(\u(t))+\beta \mathcal{C}(\u(s)) , \psi\big\rangle +  ( \f, \psi),
			\end{align}
			\item [(ii)] the initial data is satisfied in the following sense: 
			$$\lim\limits_{t\downarrow 0}\int_{\mathcal{O}}\u(t,\x)\psi(x)\d x=\int_{\mathcal{O}}\x(x)\psi(x)\d x, $$ for all $\psi\in\H$.  
		\end{enumerate}
		
	\end{definition}
	By using the monotonicity property (Theorems \ref{thm2.4} and \ref{thm2.3}) as well as Minty-Browder technique, the following global solvability result is established in Theorems 3.4 and 3.5, \cite{MTM} (see \cite{SNA,HR} also). 
	\begin{theorem}[\cite{MTM}]\label{D-Sol}
		For $r\geq 3$ ($r>3$, for any $\mu$ and $\beta$, and $r=3$ for $2\beta\mu\geq1$), let $\x \in \H$ and $\f\in \H$ be given. Then there exists a unique Leray-Hopf weak solution $\u(\cdot)$ to the system \eqref{D-CBF} in the sense of Definition \ref{def3.1}. 
		
		Moreover, the solution $\u\in\C([0,T];\H)$ and satisfies  the following energy equality: 
		\begin{align}\label{33}
		\|\u(t)\|_{\H}^2+2\mu\int_0^t\|\u(s)\|_{\V}^2\d s+2\beta\int_0^t\|\u(s)\|_{\wi\L^{r+1}}^{r+1}\d s=\|\x\|_{\H}^2+2\int_0^t(\f,\u(s))\d s,
		\end{align}
		for all $t\in[0,T]$. 
	\end{theorem}
	
	\subsection{Existence of global attractors} In this subsection, we establish the existence of a global attractor for the system \eqref{D-CBF}. 
	Thanks to the existence and uniqueness of weak solutions to the system \eqref{D-CBF} (Theorem \ref{D-Sol}), we can define a continuous semigroup $\{\S(t)\}_{t\geq0}$ in $\H$ by 
	\begin{align}\label{semigroup}
	\S(t)\x=\u(t), \quad t\geq0,
	\end{align}
	where $\u(\cdot)$ is the unique weak solution of \eqref{D-CBF} with $\u(0)=\x\in\H$. Next, we prove the existence of an absorbing ball in $\H$ as well as in $\V$ for the semigroup $\S(t), t\geq0$ defined on $\H$ for the 3D deterministic CBF equations \eqref{D-CBF} in periodic domains. We provide our estimates in terms of the dimensionless Grashof number,  which measures the relative strength of the forcing and viscosity and is defined as $$\G=\frac{\|\f\|_{\H}}{\mu^2\lambda_1}.$$ 
	\begin{theorem}\label{d-ab}
		For $r\geq 3$ ($r>3$, for any $\mu$ and $\beta$, and $r=3$ for $2\beta\mu\geq1$), let $\f\in\H$. Let us define $$\varrho_0^2:=\frac{\|\f\|_{\H}^2}{\mu^2\lambda_1^2}=\mu^2\G^2.$$ Then for any $\M_0>\varrho_0$ and $\u\in\H$, there exists a time $t_{\M_0}(\|\x\|_{\H})$ such that 
		\begin{align}
		\|\S(t)\x\|_{\H}\leq \M_0,\ \text{ for all }\ t\geq t_{\M_0}(\|\x\|_{\H}). 
		\end{align} That is, there exists an absorbing  set $\mathcal{B}_{\H}$ in $\H$ for the semigroup $\S(t)$. 
		
		For $r>3$, let 
		\begin{align}\label{36}
		\M_1^2:=(2\eta_3+1) \frac{\M_0^2}{\mu}+\frac{1}{\mu^2\lambda_1}\|\f\|^2_{\H},
		\end{align}
		where 
		\begin{align}\label{317}
		\eta_3=\frac{r-3}{r-1}\left[\frac{4}{\beta\mu (r-1)}\right]^{\frac{2}{r-3}}. 
		\end{align}	Then for any $\x\in\H$
		\begin{align}
		\|\S(t)\x\|_{\V}\leq \M_1,\ \text{ for all }\ t\geq t_{\M_0}(\|\x\|_{\H})+1. 
		\end{align} 
		That is,  there exists an absorbing  set $\mathcal{B}_{\V}$ in $\V$ for the semigroup $\S(t)$. Moreover, we have 
		\begin{align}\label{39}
		\limsup_{T\to\infty}\frac{1}{T}\int_t^{t+T}\|\A\u(s)\|_{\H}^2\d s\leq \left(\frac{\eta_3}{\lambda_1\mu}+1\right)\frac{2\|\f\|_{\H}^2}{\mu}.
		\end{align}
		For $r=3$ and $2\beta\mu\geq 1$ also, there exists an absorbing  set $\mathcal{B}_{\V}$ in $\V$ for the semigroup $\S(t)$ with $\M_1^2$ replaced by $\M_2^2=  \frac{\M_0^2}{\mu}+\frac{1}{\mu^2\lambda_1}\|\f\|^2_{\H}$ in \eqref{36}. 
	\end{theorem}
	\begin{proof}
		We divide the proof into the following steps:
		\vskip 0.2cm
		\noindent\textbf{Step I:} \emph{Absorbing ball in $\H$.} 
		From the energy equality \eqref{33}, we infer that $\|\u(t)\|_{\H}^2$ is an absolutely continuous function. 	Taking the inner product with $\u(\cdot)$ to first equation of \eqref{D-CBF} and using \eqref{b0}, we find 
		\begin{align*}
		\frac{1}{2}\frac{\d}{\d t}\|\u(t)\|^2_{\H} +\mu\|\u(t)\|^2_{\V}+\beta\|\u(t)\|^{r+1}_{\widetilde{\L}^{r+1}}&=(\f,\u(t)),
		\end{align*}
		for a.e. $t\in[0,T]$. Using the Cauchy-Schwarz inequality and Young's inequality, we estimate $|(\f,\u)|$ as 
		\begin{align*}
		|(\f,\u)|\leq\|\f\|_{\H}\|\u\|_{\H}\leq\frac{\mu}{2}\|\u\|_{\V}^2+\frac{1}{2\mu\lambda_1}\|\f\|_{\H}^2. 
		\end{align*}
		Thus, it is immediate that 
		\begin{align}\label{d-ab1}
		\frac{\d}{\d t}\|\u(t)\|^2_{\H} +\mu\lambda_1\|\u(t)\|^2_{\H}+2\beta\|\u(t)\|^{r+1}_{\widetilde{\L}^{r+1}}&\leq	\frac{\d}{\d t}\|\u(t)\|^2_{\H} +\mu\|\u(t)\|^2_{\V}+2\beta\|\u(t)\|^{r+1}_{\widetilde{\L}^{r+1}}\nonumber
		\\&\leq\frac{1}{\mu\lambda_1}\|\f\|^2_{\H}.
		\end{align}
		Using the classical Gronwall inequality, we deduce that 
		\begin{align}\label{d-ab2}
		\|\u(t)\|^2_{\H} \leq \|\x\|^2_{\H}e^{-\mu\lambda_1 t} + \frac{1}{\mu^2\lambda_1^2}\|\f\|^2_{\H}(1-e^{-\mu\lambda_1t}),
		\end{align}
		for all $t\geq0$.
		From the above relation it is clear that for any given $\M_0>\varrho_0$, there exists a time $t_0$, which depends only on $\M_0$ and $\|\x\|_{\H}$ such that 
		\begin{align}
		\|\u(t)\|_{\H}^2\leq\M_0^2,\ \text{ for all }\ t\geq t_{\M_0}(\|\x\|_{\H}). 
		\end{align}
		Furthermore, from \eqref{d-ab1}, we infer that
		\begin{align}\label{d-ab5}
		\mu\int_{t}^{t+\theta}\|\u(s)\|^2_{\V} \d s+2\beta\int_{t}^{t+\theta}\|\u(s)\|^{r+1}_{\widetilde{\L}^{r+1}}\d s\leq\|\u(t)\|_{\H}^2+\frac{\theta}{\mu\lambda_1}\|\f\|^2_{\H},
		\end{align}
		for all $\theta>0.$ The above equality  implies that for $t\geq t_{\M_0}(\|\x\|_{\H})$
		\begin{align}\label{314}
		\mu\int_{t}^{t+\theta}\|\u(s)\|^2_{\V} \d s+2\beta\int_{t}^{t+\theta}\|\u(s)\|^{r+1}_{\widetilde{\L}^{r+1}}\d s\leq \M_0^2+\frac{\theta}{\mu\lambda_1}\|\f\|^2_{\H}.
		\end{align}
		Similarly, we obtain 
		\begin{align}\label{315}
		\limsup_{T\to\infty}\frac{1}{T}\int_0^T\|\u(s)\|_{\V}^2\d s\leq \frac{1}{\mu^2\lambda_1}\|\f\|_{\H}^2\ \text{ and }\  \limsup_{T\to\infty}\frac{1}{T}\int_0^T\|\u(s)\|_{\wi\L^{r+1}}^{r+1}\d s\leq \frac{1}{\beta\mu\lambda_1}\|\f\|_{\H}^2. 
		\end{align}
		\vskip 0.2cm
		\noindent\textbf{Step II:} \emph{Absorbing ball in $\V$.} 
		Let us now take the inner product with $\A\u(\cdot)$ to first equation of \eqref{D-CBF} to  get 
		\begin{align}\label{d-ab8}
		\frac{1}{2}\frac{\d}{\d t}\|\u(t)\|^2_{\V} +\mu\|\A\u(t)\|^2_{\H}+\beta(\mathcal{C}(\u(t)),\A\u(t))=-(\B(\u(t)),\A\u(t))+(\f,\A\u(t)),
		\end{align}
		for a.e. $t\in[0,T]$. We consider the cases $r>3$ and $r=3$ separately. 
		\vskip 0.2 cm
		\noindent
		\textbf{Case I:}  $r>3$.
		From \eqref{3}, we have
		\begin{align}
		(\mathcal{C}(\u),\A\u)&=\||\nabla\u||\u|^{\frac{r-1}{2}}\|^2_{\H}+4\left[\frac{r-1}{(r+1)^2}\right]\|\nabla|\u|^{\frac{r+1}{2}}\|^2_{\H}.\label{d-ab9} 
		\end{align}
		Once again the	Cauchy-Schwarz and Young inequalities yield 
		\begin{align}
		|(\f, \A\u)|&\leq \|\f\|_{\H}\|\A\u\|_{\H}\leq\frac{\mu}{4}\|\A\u\|^2_{\H} + \frac{1}{\mu}\|\f\|^2_{\H}.\label{d-ab10}
		\end{align}
		We estimate $|(\B(\u),\A\u)|$ using H\"older's and Young's inequalities as 
		\begin{align}\label{d-ab12}
		|(\B(\u),\A\u)|&\leq\||\u||\nabla\u|\|_{\H}\|\A\u\|_{\H}\leq\frac{\mu}{4}\|\A\u\|_{\H}^2+\frac{1}{\mu }\||\u||\nabla\u|\|_{\H}^2. 
		\end{align}
		We  estimate the final term from \eqref{d-ab12} using H\"older's and Young's inequalities as (similarly as in \eqref{2.29})
		\begin{align}\label{d-ab13}
		&	\int_{\mathbb{T}^3}|\u(x)|^2|\nabla\u(x)|^2\d x\nonumber\\&\leq{\frac{\beta\mu}{2} }\left(\int_{\mathbb{T}^3}|\u(x)|^{r-1}|\nabla\u(x)|^2\d x\right)+\frac{r-3}{r-1}\left[\frac{4}{\beta\mu (r-1)}\right]^{\frac{2}{r-3}}\left(\int_{\mathbb{T}^3}|\nabla\u(x)|^2\d x\right).
		\end{align}
		Making use of the estimate \eqref{d-ab13} in \eqref{d-ab12}, we find
		\begin{align}\label{d-ab14}
		|(\B(\u),\A\u)|&\leq\frac{\mu}{4}\|\A\u\|_{\H}^2+\frac{\beta}{2}\||\nabla\u||\u|^{\frac{r-1}{2}}\|^2_{\H}+\eta_3\|\u\|^2_{\V}.
		\end{align}
		Using \eqref{d-ab9}-\eqref{d-ab10} and \eqref{d-ab14} in \eqref{d-ab8}, we obtain
		\begin{align}\label{d-ab15}
		&	\frac{\d}{\d t}\|\u(t)\|^2_{\V}+\mu\|\A\u(t)\|_{\H}^2+\beta\||\nabla\u(t)||\u(t)|^{\frac{r-1}{2}}\|^2_{\H}+8\beta\left[\frac{r-1}{(r+1)^2}\right]\|\nabla|\u(t)|^{\frac{r+1}{2}}\|^2_{\H}\nonumber\\&\leq 2\eta_3\|\u(t)\|^2_{\V}+\frac{2}{\mu}\|\f\|^2_{\H}.
		\end{align}
		We use the double integration trick used in \cite{Robinson} to obtain an absorbing ball in $\V$. Dropping the terms $\mu\|\A\u\|_{\H}^2+\beta\||\nabla\u(t)||\u|^{\frac{r-1}{2}}\|^2_{\H}+8\beta\left[\frac{r-1}{(r+1)^2}\right]\|\nabla|\u|^{\frac{r+1}{2}}\|^2_{\H}$ from the left hand side of the inequality \eqref{d-ab15} and then integrating from  $s$ to $t+1$, with $t\leq s<t+1$,  we find 
		\begin{align*}
		\|\u(t+1)\|_{\V}^2&\leq\|\u(s)\|_{\V}^2+ 2\eta_3\int_{s}^{t+1}\|\u(r)\|^2_{\V}\d r+\frac{2}{\mu}\|\f\|^2_{\H}\nonumber\\&\leq \|\u(s)\|_{\V}^2+ 2\eta_3\int_{t}^{t+1}\|\u(r)\|^2_{\V}\d r+\frac{2}{\mu}\|\f\|^2_{\H},
		\end{align*}
		where $\eta_3$ is defined in \eqref{317}. Let us now integrate both sides of the above inequality with respect to $s$ between $t$  and $t+1$ to obtain 
		\begin{align}\label{322}
		\|\u(t+1)\|_{\V}^2&\leq(2\eta_3+1)\int_t^{t+1}\|\u(s)\|_{\V}^2\d s+\frac{2}{\mu}\|\f\|^2_{\H}\leq (2\eta_3+1) \frac{\M_0^2}{\mu}+\frac{\theta}{\mu^2\lambda_1}\|\f\|^2_{\H},
		\end{align}
		for all $t\geq t_{\M_0}(\|\x\|_{\H})$. Integrating the inequality \eqref{d-ab15} from $t$ to $t+T$, we find 
		\begin{align*}
		\|\u(t+T)\|_{\V}^2+\mu\int_t^{t+T}\|\A\u(s)\|_{\H}^2\d s\leq\|\u(t)\|_{\V}^2+2\eta_3\int_t^{t+T}\|\u(s)\|_{\V}^2\d s+\frac{2T\|\f\|_{\H}^2}{\mu}.
		\end{align*}
		Dividing by $T$, taking the limit supremum and then using \eqref{315} and \eqref{322}, we finally obtain \eqref{39}. 
		\vskip 0.2 cm
		\noindent 
		\textbf{Case II:} $r=3$ and $2\beta\mu\geq1$. Using \eqref{3}, the Cauchy-Schwarz and Young inequalities, we find 
		\begin{align}
		|(\B(\u),\A\u)|&\leq\||\u||\nabla\u|\|_{\H}\|\A\u\|_{\H}\leq\frac{1}{4\beta}\|\A\u\|_{\H}^2+\beta\||\u||\nabla\u|\|_{\H}^2,\label{d-ab18}\\
		(\mathcal{C}(\u),\A\u)&=\||\nabla\u||\u|\|^2_{\H}+\frac{1}{2}\|\nabla|\u|^2\|^2_{\H},\label{d-ab19}\\
		|(\f, \A\u)|&\leq \|\f\|_{\H}\|\A\u\|_{\H}\leq\frac{\mu}{2}\|\A\u\|^2_{\H} + \frac{1}{2\mu}\|\f\|^2_{\H}.\label{d-ab20}
		\end{align}
		Using \eqref{d-ab18}-\eqref{d-ab20} in \eqref{d-ab8}, we obtain
		\begin{align}\label{d-ab21}
		\frac{\d}{\d t}\|\u(t)\|^2_{\V}+\bigg(\mu-\frac{1}{2\beta}\bigg)\|\A\u(t)\|^2_{\H}+\beta\|\nabla|\u|^2\|^2_{\H} &\leq\frac{1}{\mu}\|\f\|^2_{\H}.
		\end{align}
		For $2\beta\mu\geq 1$, performing calculations similar to the case of $r>3$, we have 
		\begin{align}\label{d-ab22}
		\|\u(t)\|^2_{\V}\leq  \frac{\M_0^2}{\mu}+\frac{\theta}{\mu^2\lambda_1}\|\f\|^2_{\H}=:\M_2^2,
		\end{align}
		for all $t\geq t_{\M_0}(\|\x\|_{\H})$, and hence the conclusions of the Theorem follows.  
	\end{proof}
	Thanks to the compactness of $\V$ in $\H$, the operators $\S(t)$ are uniformly compact. From Theorem \ref{d-ab} and the abstract theory of global attractors (Theorem 1.1, Chapter I, \cite{R.Temam}), we immediately conclude the following result.
	\begin{theorem}\label{Det-Attractor}
		The dynamical system associated with the 3D deterministic CBF equations \eqref{D-CBF} possesses an attractor $\mathcal{A}$ that is compact, connected, and maximal in $\H$. $\mathcal{A}$ attracts the bounded sets of $\H$ and $\mathcal{A}$ is also maximal among the functional invariant sets bounded in $\H$.
	\end{theorem}

	\section{Random attractors for the 3D Stochastic CBF equations}\label{sec4}\setcounter{equation}{0}
	In this section, we establish the existence of a random attractor for the 3D stochastic CBF equations  perturbed by a small additive smooth Gaussian noise. 
	
	\subsection{Abstract formulation}\label{sub4}
	Taking orthogonal projection $\mathcal{P}$ onto the first equation in \eqref{2}, we obtain the following abstract formulation of the 3D SCBF equations: 
	\begin{equation}\label{S-CBF}
	\left\{
	\begin{aligned}
	\d\u_{\varepsilon}(t)+\{\mu \A\u_{\varepsilon}(t)+\B(\u_{\varepsilon}(t))+\beta \mathcal{C}(\u_{\varepsilon}(t))\}\d t&=\f \d t + \varepsilon\d\text{W}(t), \ \ \ t\geq 0, \\ 
	\u_{\varepsilon}(0)&=\x,
	\end{aligned}
	\right.
	\end{equation}
	for $r\geq 3$ ($r>3$, for any $\mu$ and $\beta$, and $r=3$ for $2\beta\mu\geq1$) and $\varepsilon\in(0,1]$, where we assume that $\x\in \H,\ \f\in \H$ and $\text W(t), \ t\in \R,$ is a two-sided cylindrical Wiener process  with its Reproducing Kernel Hilbert Space (RKHS \cite{DZ1})  $\K$ satisfying Assumption \ref{assump} below defined on some filtered probability space $(\Omega, \mathscr{F}, (\mathscr{F}_t)_{t\in \R}, \mathbb{P})$.
	\begin{assumption}\label{assump}
		$ \K \subset \V \cap \dot{\H}^{2}_p (\mathbb{T}^3)$ is a Hilbert space such that for some $\delta\in (0, 1/2),$
		\begin{align}\label{A1}
		\A^{-\delta} : \K \to \V \cap \dot{\H}^{2}_p (\mathbb{T}^3) \ \ \  \text{is}\ \gamma \text{-radonifying.}
		\end{align}
	\end{assumption}
	\begin{remark}
		1. Let $\mathcal{H}$ be a real separable Hilbert space, let $\{w_k\}$ be an orthonormal basis of $\mathcal{H}$, and $\{\beta_k\}$ be a system of independent normal real-valued random variables defined on a probability space $(\Omega,\mathscr{F},\mathbb{P})$. For a real Banach space $(\mathrm{E},\|\cdot\|_{\mathrm{E}})$, a bounded linear operator $\Psi:\mathcal{H}\to\mathrm{E}$ is called $\gamma$-radonifying if and only if the series $\sum\limits_{k=1}^{\infty}\beta_k\Psi w_k$ converges in $\mathrm{L}^2(\Omega,\mathscr{F},\mathbb{P};\mathrm{E})$. The set of all $\gamma$-radonifying operators from $\mathcal{H}$ into $\mathrm{E}$ is denoted by $\gamma(\mathcal{H};\mathrm{E})$, and the space $(\gamma(\mathcal{H};\mathrm{E}),\|\cdot\|_{\gamma(\mathcal{H};\mathrm{E})})$, where $\|\Psi\|_{\gamma(\mathcal{H};\mathrm{E})}:=\left\{\E\left\|\sum\limits_{k=1}^{\infty}\beta_k\Psi w_k\right\|_{\mathrm{E}}^2\right\}^{1/2}$ is the norm defined on $\gamma(\mathcal{H};\mathrm{E})$, is a separable Banach space. If $\mathrm{E}$ is a separable Hilbert space then we know that $\mathcal{L}_2(\mathcal{H};\mathrm{E})$,  the class of all Hilbert-Schmidt operators from $\mathcal{H}$ into $\mathrm{E}$ endowed with the norm $\|\Psi\|_{\mathcal{L}_2(\mathcal{H};\mathrm{E})}=\left(\sum\limits_{k=1}^{\infty}\|\Psi w_k\|_{\mathrm{E}}^2\right)^{1/2}$,  is same as that of $\gamma(\mathcal{H};\mathrm{E})$. 
		
		2. Since $\D(\A)=\V \cap \dot{\H}_p^{2} (\mathbb{T}^3)$, one can reformulate Assumption \ref{assump}  in the following way also (see \cite{BL1}). $\mathrm{K}$ is a Hilbert space such that $\mathrm{K}\subset\D(\A)$ and, for some $\delta\in(0,1/2)$, the map \begin{align}\label{34}
		\A^{-\delta-1} : \mathrm{K} \to \H \   \text{ is }\ \gamma \text{-radonifying.}
		\end{align}
		The condition  \eqref{34} also says that the mapping $\A^{-\delta-1} : \mathrm{K} \to \H$  is Hilbert-Schmidt.	Since $\mathbb{T}^3$ is a periodic domain, then $\A^{-s}:\H\to\H$ is Hilbert-Schmidt if and only if $\sum_{j=1}^{\infty} \lambda_j^{-2s}<\infty,$ where  $\A w_j=\lambda_j w_j, j\in \N$ and $w_j$ is an orthogonal basis of $\H$. In periodic domains, we know that $\lambda_j\sim j^{2/3},$ for large $j$ (growth of eigenvalues, \cite{FMRT}) and hence $\A^{-s}$ is Hilbert-Schmidt if and only if $s>\frac{3}{4}.$ In other words, with $\mathrm{K}=\D(\A^{s+1}),$ the embedding $\mathrm{K}\hookrightarrow\V\cap\dot{\H}^{2}_p(\mathbb{T}^3)$ is $\gamma$-radonifying if and only if $s>\frac{3}{4}.$ Thus, Assumption \ref{assump} is satisfied for any $\delta>0.$ In fact, the condition \eqref{A1} holds if and only if the operator $\A^{-(s+1+\delta)}:\H\to \V\cap\dot{\H}^{2}_p(\mathbb{T}^3)$ is $\gamma$-radonifying. We emphasize that the requirement of  $\delta<\frac{1}{2}$ in Assumption \ref{assump} is necessary because we need (see subsection \ref{O_up}) the corresponding Ornstein-Uhlenbeck process has to take values in $\V\cap\dot{\H}^{2}_p(\mathbb{T}^3)$.
	\end{remark}
	\subsection{Solution to stochastic CBF equations }
	In this subsection, we provide the definition of a pathwise unique strong solution in the probabilistic sense to the system \eqref{S-CBF} and discuss about the existence and uniqueness of such a solution. 
	\begin{definition}
		Let $\x \in \H$,  $\f\in \H$ and $\mathrm{W}(t), t\in \R$ be two sided Wiener process in $\H$ with its Reproducing Kernel Hilbert Space (RKHS) $\mathrm{K}$. An $\H$-valued $(\mathscr{F}_t)_{t\geq 0}$-adapted stochastic process $\u_{\varepsilon}(t), \ t\geq 0,$ is called a strong solution to the system \eqref{S-CBF} if the following conditions are satisfied:
		\begin{itemize}
			\item [(i)] the process $\u_{\varepsilon}\in \mathrm{L}^{2}(\Omega; \mathrm{L}^{\infty}(0, T; \H) \cap \mathrm{L}^{2}(0, T; \V))\cap \mathrm{L}^{r+1}(\Omega; \mathrm{L}^{r+1}(0,T;\widetilde{\L}^{r+1})$  with $\mathbb{P}$-a.s., trajectories in $\mathrm{C}([0, T]; \H) \cap \mathrm{L}^{2}(0, T; \V)\cap\mathrm{L}^{r+1}(0,T;\widetilde{\L}^{r+1}),$ 
			\item [(ii)] the following equality holds for every $t\in[0,T]$ and for any $\psi\in \V\cap\widetilde{\L}^{r+1}$, $\mathbb{P}$-a.s.  
			\begin{align}\label{W-SCBF}
			(\u_{\varepsilon}(t), \psi) &= (\x, \psi) - \int_{0}^{t} \langle \mu \A\u_{\varepsilon}(s)+\B(\u_{\varepsilon}(s))+\beta \mathcal{C}(\u_{\varepsilon}(s)) , \psi\rangle \d s  +  \int_{0}^{t}( \f , \psi ) \d s  \nonumber\\&\quad+ \varepsilon \int_{0}^{t} (\d\mathrm{W}(s), \psi).
			\end{align}
		\end{itemize}
	\end{definition} 
	\begin{theorem}[\cite{MTM1}]
		Suppose that Assumption \ref{assump} is satisfied. Let $\x \in \H$, $r\geq 3$ ($r>3$, for any $\mu$ and $\beta$, and $r=3$ for $2\beta\mu\geq1$), $\f\in \H$ and $\mathrm{W}(t), t\in \R$ is a two sided Wiener process in $\H$ with its Reproducing Kernel Hilbert Space (RKHS) $\mathrm{K}$. Then there exists a pathwise unique  strong solution $\u_{\varepsilon}(\cdot)$ to the system \eqref{S-CBF}.
	\end{theorem}
	
	\subsection{Ornstein-Uhlenbeck process}\label{O_up}

	We define Ornstein-Uhlenbeck processes under Assumption \ref{assump} and discuss about its properties in this subsection (for more details see section 3, \cite{KM1}). 
	
	Let us define $\mathrm{X} := \V \cap \dot{\H}^{2}_p (\mathbb{T}^3)$ and let $\mathrm{E}$ be the completion of $\A^{-\delta}\mathrm{X}$ with respect to the image norm $\|x\|_{\mathrm{E}}  = \|\A^{-\delta} x\|_{\mathrm{X}} , \ \text{for } \ x\in \mathrm{X}, \text{where } \|\cdot\|_{\mathrm{X}} = \|\cdot\|_{\V} + \|\cdot\|_{\dot{\H}^{2}_p }.$ For $\xi \in(0, 1/2),$ we define
	\begin{align*}
	\C^{\xi}_{1/2} (\mathbb{R}, \mathrm{E}) &= \left\{ \omega \in \C(\mathbb{R}, \mathrm{E}) : \omega(0)=0,\  \sup_{t\neq s \in \mathbb{R}} \frac{\|\omega(t) - \omega(s)\|_{\mathrm{E}}}{|t-s|^{\xi}(1+|t|+|s|)^{1/2}} < \infty \right\},\\
	\Omega(\xi, \mathrm{E})&=\text{the closure of } \{ \omega \in \C^\infty_0 (\mathbb{R}, \mathrm{E}) : \omega(0) = 0 \} \ \text{ in } \ \C^{\xi}_{1/2} (\mathbb{R},\mathrm{E}).
	\end{align*}
	The space $\Omega(\xi, \mathrm{E})$ is a separable Banach space. Let us denote $\mathscr{F}$ for the Borel $\sigma$-algebra on $\Omega(\xi, \mathrm{E}).$ For $\xi\in (0, 1/2)$, there exists a Borel probability measure $\mathbb{P}$ on $\Omega(\xi, \mathrm{E})$ (see \cite{Brze}). For $t\in \mathbb{R},$ let $\mathscr{F}_t := \sigma \{ w_s : s \leq t \},$ where $w_t$ is the canonical process defined by the elements of  $\Omega(\xi, \mathrm{E}).$ Then there exists a family $\{\text{W}(t)\}_{t\in \mathbb{R}}$, which is $\mathrm{K}$-cylindrical Wiener process on a filtered probability space $(\Omega(\xi, \mathrm{E}), \mathscr{F}, (\mathscr{F}_t)_{t \in \mathbb{R}} , \mathbb{P})$.
	
	On the space $\Omega(\xi, \mathrm{E}),$ we consider a flow $\theta = (\theta_t)_{t\in \mathbb{R}}$ defined by
	$$ \theta_t \omega(\cdot) = \omega(\cdot + t) - \omega(t), \  \omega\in \Omega(\xi, \mathrm{E}), \  t\in \mathbb{R}.$$ 
	\begin{proposition}[Proposition 6.10, \cite{BL}]\label{SOUP}
		The process $\z_{\alpha}(t), \ t\in \mathbb{R},$ is stationary Ornstein-Uhlenbeck process on $(\Omega(\xi, \mathrm{E}), \mathscr{F}, \mathbb{P})$ . It is a solution of the equation 
		\begin{align}\label{OUPe}
		\d\z_{\alpha}(t) + (\mu \A + \alpha I)\z_{\alpha}(t) \d t = \d\mathrm{W}(t), \  t\in \mathbb{R},
		\end{align}
		that is, for all $t\in \mathbb{R},$ 
		\begin{align}\label{oup}
		\z_\alpha (t) = \int_{-\infty}^{t} e^{-(t-s)(\mu \A + \alpha I)} \d\mathrm{W}(s),
		\end{align}
		$\mathbb{P}$-a.s.,	where the integral is the It\^o integral on the M-type 2 Banach space $\mathrm{X}$ in the sense of \cite{Brze1}. 
		In particular, for some $C$ depending on $\mathrm{X}$,
		\begin{align}\label{E-OUP}
		\mathbb{E}\left[\|\z_{\alpha} (t)\|^2_{\mathrm{X}} \right]\leq C \int_{0}^{\infty}  e^{-2\alpha s} \|e^{-\mu s \A}\|^2_{\gamma(\mathrm{K},\mathrm{X})} \d s.
		\end{align} 
		Moreover, $\mathbb{E}\left[\|\z_{\alpha} (t)\|^2_{\mathrm{X}}\right]\to 0$ as $\alpha \to \infty.$
		
	\end{proposition}
	\begin{remark}\label{stationary}
		By Proposition 4.1, \cite{KM}, we obtain the following result for the Ornstein-Uhlenbeck process given in Proposition \ref{SOUP}:
		\begin{align}\label{O-U_conti}
		\z_\alpha\in\mathrm{L}^{q} (a, b; \mathrm{X}),
		\end{align}
		where $q\in [1, \infty].$
	\end{remark}
	\subsection{Metric dynamical system}
	Since $\z_{\alpha}(t)$ is a Gaussian random vector, by the Burkholder inequality (see \cite{Ondrejat}), for each $p\geq 2,$ there exists a constant $C_p > 0$ such that 
	\begin{align*}
	\mathbb{E}[\|\z_{\alpha}(t)\|^p_{\mathrm{X}} ]\leq C_p (\mathbb{E}[\|\z_{\alpha}(t)\|^2_{\mathrm{X}}])^{p/2}.
	\end{align*}	
	In particular, for $p=4$, we obtain 
	\begin{align*}
	\mathbb{E}[\|\z_{\alpha}(t)\|^4_{\mathrm{X}} ]\leq C (\mathbb{E}[\|\z_{\alpha}(t)\|^2_{\mathrm{X}}])^{2}.
	\end{align*}
	Moreover, 
	\begin{align}\label{alpha_con}
	\mathbb{E}[\|\z_{\alpha}(t)\|^{4}_{\mathrm{X}}] \  \to   0\ \text{ as } \ \alpha\to \infty.
	\end{align}
	By Proposition \ref{SOUP}, the process $\z_{\alpha}(t), \ t\in \R $ is an $\mathrm{X}$-valued stationary and ergodic. Hence, by the strong law of large numbers (see \cite{DZ}), we have 
	\begin{align}\label{SLLN}
	\lim_{t \to \infty} \frac{1}{t} \int_{-t}^{0} \|\z_{\alpha}(s)\|^{4}_{\mathrm{X}} ds = \mathbb{E} [\|\z_{\alpha}(0)\|^{4}_{\mathrm{X}}], \ \ \mathbb{P}\text{-a.s. on } \C^{\xi}_{1/2}(\R; \mathrm{X}).
	\end{align}
	Therefore by Proposition \ref{SOUP}, we find $\alpha_0$ such that 
	\begin{align}\label{Bddns}
	\mathbb{E}[\|\z_{\alpha} (0)\|^{4}_{\mathrm{X}}] \leq \frac{\mu^4\lambda_1}{432},
	\end{align}  
	for all $\alpha \geq \alpha_0,$  where $\lambda_1$ is the Poincar\'e constant.
	
	Let us denote 
	$$\Omega_{\alpha}(\xi, \mathrm{E})=\{\omega\in \Omega(\xi, \mathrm{E}):\text{the equality \eqref{SLLN} holds true}\}.$$
	Therefore, we fix $\xi \in (\delta, 1/2)$ and set $$\Omega := \hat{\Omega}(\xi, \mathrm{E}) = \bigcap^{\infty}_{n=0} \Omega_{n}(\xi, \mathrm{E}).$$
	\begin{proposition}\label{m-DS}
		The quadruple $(\Omega, \hat{\mathscr{F}}, \hat{\mathbb{P}}, \hat{\theta})$ is a metric DS, where $\hat{\mathscr{F}}$, $\hat{\mathbb{P}}$, $\hat{\theta}$ are respectively the natural restrictions of $\mathscr{F}$, $\mathbb{P}$, and $\theta$ to $\Omega.$ For each $\omega\in \Omega,$ the limit in \eqref{SLLN} exists.
	\end{proposition}
	The following result is the important consequence of \eqref{SLLN} and \eqref{Bddns}.
	\begin{corollary}\label{Bddns1}
		For each $\omega\in \Omega,$ there exists $t_0=t_0 (\omega) \geq 0 $ such that 
		\begin{align}\label{Bddns2}
		\frac{216}{\mu^3} \int_{-t}^{0} \|\z_{\alpha}(s)\|^{4}_{\V} \d s \leq	\frac{216}{\mu^3} \int_{-t}^{0} \|\z_{\alpha}(s)\|^{4}_{\mathrm{X}} \d s \leq \frac{\mu \lambda_1 t}{2}, \ \  t\geq t_0.
		\end{align}
	\end{corollary}
	\subsection{Random dynamical system}
	Let us recall that Assumption \ref{assump} is satisfied and that $\delta$ has the property stated there. We also take  fixed $\mu, \beta > 0$ and some parameter $\alpha\geq 0$. We also fix $\xi \in (\delta, 1/2)$.
	
	Denote by $\v_{\varepsilon}^{\alpha}(t)=\u_{\varepsilon}(t) - \varepsilon\z_{\alpha}(\omega)(t)$, then $\v_{\varepsilon}^{\alpha}(t)$ satisfies the following abstract random dynamical system:
	\begin{equation}\label{cscbf}
	\left\{
	\begin{aligned}
	\frac{\d\v_{\varepsilon}(t)}{\d t} &= -\mu \A\v_{\varepsilon}(t) - \B(\v_{\varepsilon}(t) +\varepsilon \z(t)) - \beta \mathcal{C}(\v_{\varepsilon}(t) +\varepsilon \z(t)) +\varepsilon \alpha \z(t) + \f, \\
	\v_{\varepsilon}(0)&= \x -\varepsilon \z_{\alpha}(0).
	\end{aligned}
	\right.
	\end{equation}
	Because $\z_{\alpha}(\omega) \in \C_{1/2} (\mathbb{R}, \mathrm{X}),  \z_{\alpha}(\omega)(0)$ is well defined element of $\H$.
	In what follows, we give the definition of weak solution of the system \eqref{cscbf}.
	\begin{definition}
		Assume that $\x\in \H$ and $\f\in \H$. Let $T>0$ be any fixed time, the function $\v_{\varepsilon}(\cdot)$ is called a weak solution of the problem \eqref{cscbf} on time interval $[0, T]$, if $$\v_{\varepsilon}\in \mathrm{C}([0,T];\H)\cap\mathrm{L}^2(0, T;\V)\cap\mathrm{L}^{r+1}(0,T;\widetilde{\L}^{r+1})$$ and it satisfies 
		\begin{itemize}
			\item [(i)] for any $\psi\in \V\cap\widetilde{\L}^{r+1},$ 
			\begin{align}\label{W-CSCBF}
			\bigg\langle\frac{\d\v_{\varepsilon}(t)}{\d t}, \psi\bigg\rangle =  - \big\langle \mu \A\v_{\varepsilon}(t)+\B(\v_{\varepsilon}(t)+\varepsilon\z(t))+\beta \mathcal{C}(\v_{\varepsilon}(s)+\varepsilon\z(t)) , \psi\big\rangle +  ( \f +\varepsilon \z(t), \psi).
			\end{align}
			\item [(ii)] $\v_{\varepsilon}(\cdot)$ satisfies the following initial data$$\v_{\varepsilon}(0)=\x-\varepsilon\z_{\alpha}(0).$$
		\end{itemize}
	\end{definition}
	Since, $\u_{\varepsilon}(\cdot)$ is the unique solution to the problem \eqref{S-CBF} and $\z_{\alpha}(\cdot)$ is the unique solution to the problem \eqref{OUPe}, one can easily obtain a unique solution $\v_{\varepsilon}^{\alpha}(t)$  to the problem \eqref{cscbf}. In the next theorem, we take $\f$ is depending on $t$. 
	\begin{theorem}\label{RDS_Conti}
		For $r\geq3$ ($r>3$, for any $\mu$ and $\beta$, and $r=3$ for $2\beta\mu\geq1$),	assume that, for some $T >0$ fixed, $\x_n \to \x$ in $\H$, 
		\begin{align*}
		\z_n \to \z \ \text{ in }\ \mathrm{L}^{\infty} (0, T; \V)\cap\mathrm{L}^{r+1}(0,T; \D(\A))\  \text{ and }\  \f_n \to \f \ \text{ in }\ \mathrm{L}^2 (0, T; \H). 
		\end{align*}
		Let us denote by $\v_{\varepsilon}(t, \z)\x$ for the solution of the problem \eqref{cscbf} and by $\v_{\varepsilon}(t, \z_n)\x_n$ for the solution of the problem \eqref{cscbf} with $\z, \f, \x$ being replaced by $\z_n, \f_n, \x_n$. Then $$\v_{\varepsilon}(\cdot, \z_n)\x_n \to \v_{\varepsilon}(\cdot, \z)\x \ \text{ in }\ \mathrm{C}([0,T];\H)\cap\mathrm{L}^2 (0, T;\V)\cap\mathrm{L}^{r+1}(0,T;\wi\L^{r+1}).$$
		In particular, $\v_{\varepsilon}(T, \z_n)\x_n \to \v_{\varepsilon}(T, \z)\x$ in $\H$.
	\end{theorem}	
	\begin{proof}
		In order to simplify the proof, we introduce the following notations:
		\begin{align*}\v_n (t) &= \v_{\varepsilon}(t, \z_n)\x_n, \ \  \v(t) = \v_{\varepsilon}(t, \z)\x,\ \   \y_n (t)= \v_{\varepsilon}(t, \z_n)\x_n - \v_{\varepsilon}(t, \z)\x, \ \  t\in[0, T],\\
		\hat{ \z}_n &= \z_n - \z, \ \  \hat{ \f}_n = \f_n - \f.\end{align*}
		It is easy to see that $\y_n(\cdot)$ solves the following initial value problem:
		\begin{equation}\label{cscbf_n}
		\left\{
		\begin{aligned}
		\frac{\d\y_n(t)}{\d t} &= -\mu \A\y_n (t) - \B(\v_n (t) + \varepsilon\z_n (t)) + \B(\v(t) + \varepsilon\z(t)) - \beta \mathcal{C}(\v_n (t) + \varepsilon\z_n (t)) \\& \ \ \ + \beta \mathcal{C}(\v(t) +\varepsilon \z(t)) + \varepsilon\alpha \hat{ \z}_n + \hat{ \f}, \\
		\y_n(0)&= \x_n - \x.
		\end{aligned}
		\right.
		\end{equation}
		Taking the inner product with $\y_n(\cdot)$ to the first equation of \eqref{cscbf_n}, we get 
		\begin{align}\label{Energy_esti_n_1}
		&	\frac{1}{2} \frac{\d}{\d t}\|\y_n (t)\|^2_{\H}  \nonumber\\&=-  \mu \|\y_n(t)\|^2_{\V} - \big\langle\B(\v_n(t)+\varepsilon\z_n(t))- \B(\v(t)+\varepsilon\z(t)), \y_n(t)\big\rangle\nonumber\\&\quad - \beta\big\langle\mathcal{C}(\v_n(t)+\varepsilon\z_n(t))-\mathcal{C}(\v(t)+\varepsilon\z(t)), \y_n(t)\big\rangle +\varepsilon \alpha (\hat{ \z}_n(t), \y_n(t)) + (\hat{ \f}(t),\y_n(t)), \nonumber\\&=-  \mu \|\y_n(t)\|^2_{\V}+\big\langle\B(\v_n(t)+\varepsilon\z_n(t)), \varepsilon\hat{ \z}_n(t)\big\rangle-\big\langle\B(\v(t)+\varepsilon\z(t)), \varepsilon\hat{ \z}_n(t)\big\rangle\nonumber\\&\ \ \ -\big\langle\B(\v_n(t)+\varepsilon\z_n(t))- \B(\v(t)+\varepsilon\z(t)), (\v_n(t)+\varepsilon\z_n(t))-(\v(t)+\varepsilon\z(t))\big\rangle\nonumber\\&\ \ \ +\beta\big\langle\mathcal{C}(\v_n(t)+\varepsilon\z_n(t))-\mathcal{C}(\v(t)+\varepsilon\z(t)), \varepsilon\hat{ \z}_n(t)\big\rangle\nonumber\\&\ \ \ - \beta\big\langle\mathcal{C}(\v_n(t)+\varepsilon\z_n(t))-\mathcal{C}(\v(t)+\varepsilon\z(t)),(\v_n(t)+\varepsilon\z_n(t))-(\v(t)+\varepsilon\z(t))\big\rangle\nonumber\\&\ \ \ + \varepsilon\alpha (\hat{ \z}_n(t), \y_n(t)) + (\hat{ \f}(t),\y_n(t)),
		\end{align}
		for a.e. $t\in[0,T]$. We consider the cases $r>3$ and $r=3$ separately. 
		\vskip 0.2 cm
		\noindent 
		\textbf{Case I:}  $r>3$.
		Using $0<\varepsilon\leq1$ and \eqref{212}, we have
		\begin{align}
		|\big\langle\B(\v_n+\varepsilon\z_n), \varepsilon\hat{ \z}_n\big\rangle|\leq&\|\v_n+\varepsilon\z_n\|_{\widetilde{\L}^{r+1}}^{\frac{r+1}{r-1}}\|\v_n+\varepsilon\z_n\|_{\H}^{\frac{r-3}{r-1}}\|\hat{ \z}_n\|_{\V}\nonumber\\\leq&C\|\v_n+\varepsilon\z_n\|_{\widetilde{\L}^{r+1}}^{\frac{r+1}{r-1}}\|\v_n+\varepsilon\z_n\|_{\V}^{\frac{r-3}{r-1}}\|\hat{ \z}_n\|_{\V},\label{Energy_esti_n_2}\\
		|\big\langle\B(\v+\varepsilon\z), \varepsilon\hat{ \z}_n\big\rangle|\leq&\|\v+\varepsilon\z\|_{\widetilde{\L}^{r+1}}^{\frac{r+1}{r-1}}\|\v+\varepsilon\z\|_{\H}^{\frac{r-3}{r-1}}\|\hat{ \z}_n\|_{\V}\nonumber\\\leq&C\|\v+\varepsilon\z\|_{\widetilde{\L}^{r+1}}^{\frac{r+1}{r-1}}\|\v+\varepsilon\z\|_{\V}^{\frac{r-3}{r-1}}\|\hat{ \z}_n\|_{\V}.\label{Energy_esti_n_3}
		\end{align}
		Using $0<\varepsilon\leq1$, \eqref{441} and \eqref{2.16}, we have
		\begin{align}\label{Energy_esti_n_4}
		&|\langle\B(\v_n+\varepsilon\z_n)-\B(\v+\varepsilon\z),(\v_n+\varepsilon\z_n)-(\v+\varepsilon\z)\rangle|\nonumber\\&\leq\frac{\mu }{4}\|(\v_n+\varepsilon\z_n)-(\v+\varepsilon\z)\|_{\V}^2+\frac{\beta}{8}\||\v+\varepsilon\z|^{\frac{r-1}{2}}((\v_n+\varepsilon\z_n)-(\v+\varepsilon\z))\|_{\H}^2\nonumber\\&\quad+\eta_1\|(\v_n+\varepsilon\z_n)-(\v+\varepsilon\z)\|_{\H}^2\nonumber\\&\leq\frac{\mu}{2}\|\y_n\|_{\V}^2+\frac{\mu}{2}\|\hat{\z}_n\|_{\V}^2+\frac{\beta}{8}\||\v+\varepsilon\z|^{\frac{r-1}{2}}((\v_n+\varepsilon\z_n)-(\v+\varepsilon\z))\|_{\H}^2\nonumber\\&\quad+2\eta_1\|\y_n\|_{\H}^2+2\eta_1\|\hat{\z}_n\|_{\H}^2\nonumber\\&\leq\frac{\mu}{2}\|\y_n\|_{\V}^2+C\|\hat{\z}_n\|_{\V}^2+\frac{\beta}{8}\||\v+\varepsilon\z|^{\frac{r-1}{2}}((\v_n+\varepsilon\z_n)-(\v+\varepsilon\z))\|_{\H}^2+2\eta_1\|\y_n\|_{\H}^2,
		\end{align}
		where $\eta_1=\frac{r-3}{\mu(r-1)}\left(\frac{16}{\beta\mu (r-1)}\right)^{\frac{2}{r-3}}.$
		Using Taylor's formula (Theorem 7.9.1, \cite{PGC}), we find 
		\begin{align}\label{Energy_esti_n_5}
		&\beta\big|\big\langle\mathcal{C}(\v_n+\varepsilon\z_n)-\mathcal{C}(\v+\varepsilon\z), \varepsilon\hat{ \z}_n\big\rangle\big|\nonumber\\&= \beta\bigg|\bigg\langle \int_{0}^{1} \big[\mathcal{C}'(\theta(\v_n+\varepsilon\z_n) + (1-\theta)(\v+\varepsilon\z))((\v_n+\varepsilon\z_n)-(\v+\varepsilon\z))\big] \d\theta ,\varepsilon\hat{ \z}_n\bigg\rangle\bigg|\nonumber\\ &\leq r\beta 2^{r-2} \||(\v_n+\varepsilon\z_n)|^{\frac{r-1}{2}}((\v_n+\varepsilon\z_n)-(\v+\varepsilon\z))\|_{\H} \|\v_n+\varepsilon\z_n\|_{\widetilde{\L}^{r+1}}^{\frac{r-1}{2}} \|\hat{\z}_n\|_{\widetilde{\L}^{r+1}}\nonumber\\&\quad+r\beta 2^{r-2} \||(\v+\varepsilon\z)|^{\frac{r-1}{2}}((\v_n+\varepsilon\z_n)-(\v+\varepsilon\z))\|_{\H} \|\v+\varepsilon\z\|_{\widetilde{\L}^{r+1}}^{\frac{r-1}{2}} \|\hat{\z}_n\|_{\widetilde{\L}^{r+1}}\nonumber\\&\leq\frac{\beta}{4} \||(\v_n+\varepsilon\z_n)|^{\frac{r-1}{2}}((\v_n+\varepsilon\z_n)-(\v+\varepsilon\z))\|^2_{\H}+C \|(\v_n+\varepsilon\z_n)\|_{\widetilde{\L}^{r+1}}^{r-1} \|\hat{\z}_n\|^2_{\widetilde{\L}^{r+1}}\nonumber\\&\quad+\frac{\beta}{8} \||(\v+\varepsilon\z)|^{\frac{r-1}{2}}((\v_n+\varepsilon\z_n)-(\v+\varepsilon\z))\|^2_{\H}+C \|(\v+\varepsilon\z)\|_{\widetilde{\L}^{r+1}}^{r-1} \|\hat{\z}_n\|^2_{\widetilde{\L}^{r+1}}
		\end{align}
		Making use of \eqref{MO_c}, we obtain 
		\begin{align}\label{Energy_esti_n_6}
		-&\beta\big\langle\mathcal{C}(\v_n+\varepsilon\z_n)-\mathcal{C}(\v+\varepsilon\z),(\v_n+\varepsilon\z_n)-(\v+\varepsilon\z)\big\rangle \nonumber\\&\leq-\frac{\beta}{2}\||\v_n+\varepsilon\z_n|^{\frac{r-1}{2}}((\v_n+\varepsilon\z_n)-(\v+\varepsilon\z))\|_{\H}^2-\frac{\beta}{2}\||\v+\varepsilon\z|^{\frac{r-1}{2}}((\v_n+\varepsilon\z_n)-(\v+\varepsilon\z))\|_{\H}^2.
		\end{align}
		Using H\"older's and Young's inequalities, we have
		\begin{align}
		\varepsilon\alpha|(\hat{ \z}_n,  \y_n)|& \leq \alpha \|\y_n\|_{\H} \|\hat{ \z}_n\|_{\H} \leq \frac{1 }{2} \|\y_n\|^2_{\H} + \frac{\alpha^2}{2} \|\hat{ \z}_n\|^2_{\H} \leq \frac{1}{2} \|\y_n\|^2_{\H} + \frac{\alpha^2C}{2} \|\hat{ \z}_n\|^2_{\V},\label{Energy_esti_n_7}\\
		|(\hat{ \f}_n,  \y_n)|& \leq  \|\y_n\|_{\H} \|\hat{ \f}_n\|_{\H} \leq \frac{1}{2} \|\y_n\|^2_{\H} + \frac{1}{2} \|\hat{ \f}_n\|^2_{\H}.\label{Energy_esti_n_8}
		\end{align}
		Combining  \eqref{Energy_esti_n_2}-\eqref{Energy_esti_n_8}, substituting it  in \eqref{Energy_esti_n_1}, and then using \eqref{a215}, we deduce that 
		\begin{align}\label{Energy_esti_n_9}
		&\frac{\d}{\d t} \|\y_n(t)\|_{\H}^2 + \mu\|\y_n(t)\|_{\V}^2+\frac{\beta}{2^{r-1}}\|\y_n(t)+\e\hat\z_n(t)\|_{\wi\L^{r+1}}^{r+1} \nonumber\\&\leq C\bigg\{\|\v_n(t)+\varepsilon\z_n(t)\|_{\widetilde{\L}^{r+1}}^{\frac{r+1}{r-1}}\|\v_n(t)+\varepsilon\z_n(t)\|_{\V}^{\frac{r-3}{r-1}}+\|\v(t)+\varepsilon\z(t)\|_{\widetilde{\L}^{r+1}}^{\frac{r+1}{r-1}}\|\v(t)+\varepsilon\z(t)\|_{\V}^{\frac{r-3}{r-1}}\bigg\}\nonumber\\&\qquad\times\|\hat{\z}_n(t)\|_{\V}+C\|\hat{\z}_n(t)\|_{\V}^2+\wi\eta_1\|\y_n(t)\|_{\H}^2+C\bigg\{\|(\v_n(t)+\varepsilon\z_n(t))\|_{\widetilde{\L}^{r+1}}^{r-1}\nonumber\\&\quad+\|(\v(t)+\varepsilon\z(t))\|_{\widetilde{\L}^{r+1}}^{r-1}\bigg\}\|\hat{\z}_n(t)\|^2_{\widetilde{\L}^{r+1}} + \|\hat{ \f}_n(t)\|^2_{\H},
		\end{align}
		where $\wi\eta_1=2(2\eta_1+1),$ for a.e. $t\in[0,T]$. By integrating the above inequality from $0$ to $t$, for $t\in[0,T]$, we get 
		\begin{align}\label{Energy_esti_n_10}
		&	\|\y_n(t)\|^2_{\H} + \mu \int_{0}^{t} \|\y_n(s) \|^2_{\V}\d s+\frac{\beta}{2^{r-1}}\int_0^t\|\y_n(s)+\e\hat\z_n(s)\|_{\wi\L^{r+1}}^{r+1}\d s\nonumber\\&\leq \|\y_n(0)\|^2_{\H} + \int_{0}^{t} \wi\eta_1 \|\y_n(s)\|^2_{\H} \d s + \int_{0}^{t} \text{K}_n(s)  \d s, 
		\end{align}
		for all $t\in[0,T]$,	where
		\begin{align*}
		\text{K}_n& = C\bigg\{\|\v_n+\varepsilon\z_n\|_{\widetilde{\L}^{r+1}}^{\frac{r+1}{r-1}}\|\v_n+\varepsilon\z_n\|_{\V}^{\frac{r-3}{r-1}}+\|\v+\varepsilon\z\|_{\widetilde{\L}^{r+1}}^{\frac{r+1}{r-1}}\|\v+\varepsilon\z\|_{\V}^{\frac{r-3}{r-1}}\bigg\}\|\hat{\z}_n\|_{\V}\nonumber\\&\quad+C\bigg\{\|\v_n+\varepsilon\z_n\|_{\widetilde{\L}^{r+1}}^{r-1}+\|\v+\varepsilon\z\|_{\widetilde{\L}^{r+1}}^{r-1}\bigg\} \|\hat{\z}_n\|^2_{\widetilde{\L}^{r+1}}+C \|\hat{ \z}_n\|^2_{\V}+\|\hat{ \f}_n\|^2_{\H}.
		\end{align*}
		for a.e. $t\in[0,T]$. Then by the classical Gronwall inequality, we arrive at 
		\begin{align}\label{Energy_esti_n_11}
		\|\y_n(t)\|^2_{\H}\leq \bigg(\|\y_n(0)\|^2_{\H}+ \int_{0}^{T} \text{K}_n(t)\d t\bigg) e^{\wi\eta_1 T}, 
		\end{align}
		for all $t\in [0,T]$.	On the other hand, we have
		\begin{align*}
		&\int_{0}^{T}\text{K}_n(t)\d t\\ =&\int_{0}^{T} \bigg[C\bigg\{\|\v_n(t)+\varepsilon\z_n(t)\|_{\widetilde{\L}^{r+1}}^{\frac{r+1}{r-1}}\|\v_n(t)+\varepsilon\z_n(t)\|_{\V}^{\frac{r-3}{r-1}}+\|\v(t)+\varepsilon\z(t)\|_{\widetilde{\L}^{r+1}}^{\frac{r+1}{r-1}}\|\v(t)+\varepsilon\z(t)\|_{\V}^{\frac{r-3}{r-1}}\bigg\}\\&\qquad\times\|\hat{\z}_n(t)\|_{\V}+C\bigg\{\|(\v_n(t)+\varepsilon\z_n(t))\|_{\widetilde{\L}^{r+1}}^{r-1}+\|(\v(t)+\varepsilon\z(t))\|_{\widetilde{\L}^{r+1}}^{r-1} \bigg\}\|\hat{\z}_n(t)\|^2_{\widetilde{\L}^{r+1}}\\&\quad+C \|\hat{ \z}_n(t)\|^2_{\V}+ \|\hat{ \f}_n(t)\|^2_{\H}\bigg]\d t\\
		\leq& CT^{1/2}\bigg[\|\v_n+\varepsilon\z_n\|^{\frac{r+1}{r-1}}_{\mathrm{L}^{r+1}(0,T;\widetilde{\L}^{r+1})}\|\v_n+\varepsilon\z_n\|^{\frac{r-3}{r-1}}_{\mathrm{L}^{2}(0,T;\V)} +\|\v+\varepsilon\z\|^{\frac{r+1}{r-1}}_{\mathrm{L}^{r+1}(0,T;\widetilde{\L}^{r+1})}\|\v+\varepsilon\z\|^{\frac{r-3}{r-1}}_{\mathrm{L}^{2}(0,T;\V)}\bigg]\\&\qquad\times\|\hat{\z}_n\|_{\mathrm{L}^{\infty}(0,T;\V)}+C\bigg[\|\v_n+\varepsilon\z_n\|^{r-1}_{\mathrm{L}^{r+1}(0,T;\widetilde{\L}^{r+1})}+\|\v+\varepsilon\z\|^{r-1}_{\mathrm{L}^{r+1}(0,T;\widetilde{\L}^{r+1})}\bigg]\|\hat{\z}_n\|^2_{\mathrm{L}^{r+1}(0,T;\widetilde{\L}^{r+1})}\\&\quad+CT\|\hat{\z}_n\|^2_{\mathrm{L}^{\infty}(0,T;\V)}+\|\hat{\f}_n\|^2_{\mathrm{L}^2(0,T;\H)}.
		\end{align*}
		Since, $\z_n\to\z$ in $\mathrm{L}^{r+1}(0,T;\D(\A))$ and by the Sobolev inequality, we have $$\|\hat{\z}_n\|_{\mathrm{L}^{r+1}(0,T;\widetilde{\L}^{r+1})}\leq C \|\hat{\z}_n\|_{\mathrm{L}^{r+1}(0,T;\D(\A))},$$ which implies that $\z_n\to\z$ in $\mathrm{L}^{r+1}(0,T;\widetilde{\L}^{r+1})$. Therefore, $\int_{0}^{T} \text{K}_n(t) \d t \to 0$ as $n\to \infty.$
		
		Since, $\|\y_n(0)\|_{\H} = \|\x_n- \x\|_{\H} \to 0 $ and $\int_{0}^{T} \text{K}_n(s) \d s \to 0$ as $n\to \infty$, by \eqref{Energy_esti_n_11} we infer that $\|\y_n(t)\|_{\H}\to 0$ as $n\to\infty$ uniformly in $t\in[0, T].$ In other words,  $$\v_{\varepsilon}(\cdot, \z_n)\x_n \to \v_{\varepsilon}(\cdot, \z)\x\ \ \text{ in }\  \C([0, T]; \H).$$
		From the inequality \eqref{Energy_esti_n_10}, we also have
		\begin{align*}
		&	\mu \int_{0}^{T} \|\y_n(s)\|^2_{\V}\d s+\frac{\beta}{2^{2r-1}}\int_0^T\|\y_n(t)\|_{\wi\L^{r+1}}^{r+1}\d t\nonumber\\&\leq 	\mu \int_{0}^{T} \|\y_n(s)\|^2_{\V}\d s+\frac{\beta}{2^{r-1}}\int_0^T\|\y_n(t)+\e\hat\z_n(t)\|_{\wi\L^{r+1}}^{r+1}\d t+\frac{\beta}{2^{r-1}}\int_0^T\|\hat\z_n(t)\|_{\wi\L^{r+1}}^{r+1}\d t\nonumber\\
		& \leq \|\y_n(0)\|^2_{\H} +\wi\eta_1T \sup_{s  \in [0, T]}\|\y_n(s)\|^2_{\H} + \int_{0}^{T} \text{K}_n(s) \d s+\frac{\beta}{2^{r-1}}\int_0^T\|\hat\z_n(t)\|_{\wi\L^{r+1}}^{r+1}\d t.
		\end{align*}
		Hence, $\int_{0}^{T} \|\y_n(s)\|^2_{\V}\d s, \int_0^T\|\y_n(t)\|_{\wi\L^{r+1}}^{r+1}\d t \to 0$ as $n\to \infty$ and therefore, $$\v_{\varepsilon}(\cdot, \z_n)\x_n \to \v_{\varepsilon}(\cdot, \z)\x\   \text{ in } \ \mathrm{L}^2(0, T; \V)\cap\mathrm{L}^{r+1}(0,T;\wi\L^{r+1}),$$ which completes the proof for $r>3$.
		\vskip 0.2cm
		\noindent
		\textbf{Case II:} $r=3$ and $2\beta\mu\geq1$.	By H\"older's inequality, \eqref{lady} and \eqref{poin}, we find 
		\begin{align}
		|\big\langle\B(\v_n+\varepsilon\z_n),\varepsilon \hat{ \z}_n\big\rangle|&\leq\|\v_n+\varepsilon\z_n\|_{\widetilde{\L}^{4}}^{2}\|\hat{ \z}_n\|_{\V}\leq C \|\v_n+\varepsilon\z_n\|_{\V}^{2}\|\hat{ \z}_n\|_{\V},\label{Energy_esti_n_12}\\
		|\big\langle\B(\v+\varepsilon\z), \varepsilon\hat{ \z}_n\big\rangle|&\leq\|\v+\varepsilon\z\|_{\widetilde{\L}^{4}}^{2}\|\hat{ \z}_n\|_{\V}\leq C\|\v+\varepsilon\z\|_{\V}^{2}\|\hat{ \z}_n\|_{\V}.\label{Energy_esti_n_13}
		\end{align}
		Using \eqref{441} and H\"older's inequality, we obtain 
		\begin{align}\label{Energy_esti_n_14}
		&|\langle\B(\v_n+\varepsilon\z_n)-\B(\v+\varepsilon\z),(\v_n+\varepsilon\z_n)-(\v+\varepsilon\z)\rangle|\nonumber\\&\leq\|(\v_n+\varepsilon\z_n)-(\v+\varepsilon\z)\|_{\V}\||\v+\varepsilon\z|((\v_n+\varepsilon\z_n)-(\v+\varepsilon\z))\|_{\H}\nonumber\\&\leq\frac{1}{2\beta}\|(\v_n+\varepsilon\z_n)-(\v+\varepsilon\z)\|_{\V}^2+\frac{\beta}{2}\||\v+\varepsilon\z|((\v_n+\varepsilon\z_n)-(\v+\varepsilon\z))\|_{\H}^2\nonumber\\&\leq\frac{1}{2\beta}(\|(\v_n-\v)\|_{\V}+\|(\z_n-\z)\|_{\V})^2+\frac{\beta}{2}\||\v+\varepsilon\z|((\v_n+\varepsilon\z_n)-(\v+\varepsilon\z))\|_{\H}^2\nonumber\\&\leq\frac{1}{2\beta}\|\y_n\|_{\V}^2+ \frac{1}{2\beta}\|\hat{ \z}_n\|_{\V}^2+\frac{1}{\beta}\|\v_n\|_{\V}\|\hat{ \z}_n\|_{\V}+\frac{1}{\beta}\|\v\|_{\V}\|\hat{ \z}_n\|_{\V}\nonumber\\&\quad+\frac{\beta}{2}\||\v+\varepsilon\z|((\v_n+\varepsilon\z_n)-(\v+\varepsilon\z))\|_{\H}^2.
		\end{align}
		Using H\"older's inequality, \eqref{lady} and \eqref{poin}, we have
		\begin{align}\label{Energy_esti_n_15}
		|\langle\mathcal{C}(\v_n+\varepsilon\z_n)-\mathcal{C}(\v+\varepsilon\z),\varepsilon \hat{ \z}_n\rangle|&\leq\big\{\|\v_n+\varepsilon\z_n\|^3_{\L^4}+\|\v+\varepsilon\z\|^3_{\L^4}\big\}\|\hat{ \z}_n\|_{\L^4}\nonumber\\&\leq C\big\{\|\v_n\|^3_{\L^4}+\|\v\|^3_{\L^4}+\|\z_n\|^3_{\V}+\|\z\|^3_{\V}\big\}\|\hat{ \z}_n\|_{\V}.
		\end{align}
		By \eqref{MO_c}, we obtain 
		\begin{align}\label{Energy_esti_n_16}
		&-\beta\big\langle\mathcal{C}(\v_n+\varepsilon\z_n)-\mathcal{C}(\v+\varepsilon\z),(\v_n+\varepsilon\z_n)-(\v+\varepsilon\z)\big\rangle\nonumber\\&\leq-\frac{\beta}{2}\||\v_n+\varepsilon\z_n|((\v_n+\varepsilon\z_n)-(\v+\varepsilon\z))\|_{\H}^2-\frac{\beta}{2}\||\v+\varepsilon\z|((\v_n+\varepsilon\z_n)-(\v+\varepsilon\z))\|_{\H}^2.
		\end{align}
		Using \eqref{Energy_esti_n_12}-\eqref{Energy_esti_n_16} along with \eqref{Energy_esti_n_7}-\eqref{Energy_esti_n_8}, \eqref{a215} in \eqref{Energy_esti_n_1}, we deduce that 
		\begin{align}\label{Energy_esti_n_19}
		&\frac{\d}{\d t} \|\y_n(t)\|_{\H}^2 + 2\bigg(\mu-\frac{1}{2\beta}\bigg)\|\y_n(t)\|_{\V}^2 \nonumber\\&\leq C\bigg\{ \|\v_n(t)+\varepsilon\z_n(t)\|_{\V}^{2} +\|\v(t)+\varepsilon\z(t)\|_{\V}^{2}+ \|\v_n(t)\|_{\V}+\|\v(t)\|_{\V}+\|\v_n(t)\|^3_{\L^4}+\|\v(t)\|^3_{\L^4}\nonumber\\&\quad\quad+\|\z_n(t)\|^3_{\V}+\|\z(t)\|^3_{\V}\bigg\}\|\hat{ \z}_n(t)\|_{\V}+2\|\y_n(t)\|^2_{\H}+C \|\hat{ \z}_n(t)\|^2_{\V}+ \|\hat{ \f}_n(t)\|^2_{\H},
		\end{align}
		for a.e. $t\in[0,T].$ Integrating the above inequality from $0$ to $t$, $t\in[0,T]$, we get 
		\begin{align}\label{Energy_esti_n_20}
		\|\y_n(t)\|^2_{\H} + 2\bigg(\mu-\frac{1}{2\beta}\bigg) \int_{0}^{t} \|\y_n(s) \|^2_{\V}\d s\leq \|\y_n(0)\|^2_{\H} + 2\int_{0}^{t}  \|\y_n(s)\|^2_{\H}\d s + \int_{0}^{t} \widetilde{\text{K}}_n(s)\d s,  
		\end{align}
		for all $t\in[0, T],$ where
		\begin{align*}
		\widetilde{\text{K}}_n =&C\bigg\{ \|\v_n+\varepsilon\z_n\|_{\V}^{2} +\|\v+\varepsilon\z\|_{\V}^{2}+ \|\v_n\|_{\V}+\|\v\|_{\V}+\|\v_n\|^3_{\L^4}+\|\v\|^3_{\L^4}+\|\z_n\|^3_{\V}+\|\z\|^3_{\V}\bigg\}\nonumber\\&\quad\times\|\hat{ \z}_n\|_{\V}+C \|\hat{ \z}_n\|^2_{\V}+ \|\hat{ \f}_n\|^2_{\H},
		\end{align*}
		for a.e. $t\in[0,T]$.	Then, for $2\beta\mu\geq1$, by the classical Gronwall inequality, we obtain 
		\begin{align}\label{Energy_esti_n_21}
		\|\y_n(t)\|^2_{\H}\leq \bigg(\|\y_n(0)\|^2_{\H}+ \int_{0}^{T} \widetilde{\text{K}}_n(s)\d s\bigg) e^{2 T},
		\end{align}
		for all $t\in [0,T]$. On the other hand, we have
		\begin{align*}
		&\int_{0}^{T}\widetilde{\text{K}}_n(s)\d s\nonumber\\ &
		\leq\bigg\{\|\v_n+\varepsilon\z_n\|_{\mathrm{L}^{2}(0,T;\V)}+\|\v+\varepsilon\z\|_{\mathrm{L}^{2}(0,T;\V)} +T^{1/2}\|\v_n\|_{\mathrm{L}^{2}(0,T;\V)}+T^{1/2}\|\v\|_{\mathrm{L}^{2}(0,T;\V)}\nonumber\\&\quad+T^{1/4}\|\v_n\|^3_{\mathrm{L}^4(0,T;\widetilde{\L}^4)}+T^{1/4}\|\v\|^3_{\mathrm{L}^4(0,T;\widetilde{\L}^4)} +T\|\z_n\|^3_{\mathrm{L}^{\infty}(0,T;\V)}+T\|\z\|^3_{\mathrm{L}^{\infty}(0,T;\V)} \bigg\}\|\hat{\z}_n\|_{\mathrm{L}^{\infty}(0,T;\V)}\\&\quad+T\|\hat{\z}_n\|^2_{\mathrm{L}^{\infty}(0,T;\V)}+\|\hat{\f}_n\|^2_{\mathrm{L}^2(0,T;\H)}.
		\end{align*}
		Therefore, $\int_{0}^{T} \widetilde{\text{K}}_n(s)\d s \to 0$ as $n\to \infty.$	Since, $\|\y_n(0)\|_{\H} = \|\x_n- \x\|_{\H} \to 0 $ and $\int_{0}^{T} \widetilde{\text{K}}_n(s)\d s \to 0$ as $n\to \infty$ , by \eqref{Energy_esti_n_21} we infer that $\|\y_n(t)\|_{\H}\to 0$ as $n\to\infty$ uniformly in $t\in[0, T].$ In other words, $$\v_{\varepsilon}(\cdot, \z_n)\x_n \to \v_{\varepsilon}(\cdot, \z)\x\   \text{ in } \ \C([0, T]; \H).$$
		From inequality \eqref{Energy_esti_n_20}, we also have
		\begin{align*}
		2\bigg(\mu-\frac{1}{2\beta}\bigg) \int_{0}^{T} \|\y_n(t) \|^2_{\V}\d t \leq \|\y_n(0)\|^2_{\H} +2T \sup_{s  \in [0, T]}\|\y_n(s)\|^2_{\H} + \int_{0}^{T} \widetilde{\text{K}}_n(t)\d t.
		\end{align*}
		Thus, for $2\beta\mu>1$, it is immediate that  $\int_{0}^{T} \|\y_n(t)\|^2_{\V}\d t \to 0$ as $n\to \infty$ and therefore, we arrive at $$\v_{\varepsilon}(\cdot, \z_n)\x_n \to \v_{\varepsilon}(\cdot, \z)\x\  \text{ in }\ \mathrm{L}^2(0, T; \V),$$ which completes the proof for the case $r=3$.
	\end{proof}
	\begin{definition}
		We define a map $\varphi_{\varepsilon}^{\alpha} : \mathbb{R}^+ \times \Omega \times \H \to \H$ by
		\begin{align}
		(t, \omega, x) \mapsto \v_{\varepsilon}^{\alpha}(t)  + \varepsilon\z_{\alpha}(\omega)(t) \in \H.
		\end{align}
	\end{definition}
	\begin{proposition}\label{alpha_ind}
		If $\alpha_1, \alpha_2 \geq 0$, then $\varphi_{\varepsilon}^{\alpha_1} = \varphi_{\varepsilon}^{\alpha_2}.$
	\end{proposition}
	\begin{proof}
		Using inequality \eqref{fe4} (for $r>3$) and inequality \eqref{fe5} (for $r=3$), one can complete  the proof  similarly as in Proposition 4.11, \cite{KM}. 
	\end{proof}
	Using Proposition \ref{alpha_ind}, we denote $\varphi_{\varepsilon}^{\alpha}$  by $\varphi_{\varepsilon}$.
	\begin{theorem}
		$(\varphi_{\varepsilon}, \theta)$ is an RDS.
	\end{theorem}
	\begin{proof}
		All the properties with the exception of the cocycle one of an RDS follow from Theorem \ref{RDS_Conti}. Hence we only need to show that for any $\x\in \H,$
		\begin{align}\label{Cocy..}
		\varphi_{\varepsilon}(t+s, \omega)\x = \varphi_{\varepsilon}(t, \theta_s \omega)\varphi_{\varepsilon}(s, \omega)\x, \ \ t, s \in \R^+. 
		\end{align}
		Remaining proof is similar to that of Theorem 6.15, \cite{BL} and hence we omit it here.
	\end{proof}
	\subsection{Existence of random attractors}
	In this subsection, we consider the RDS $\varphi_{\varepsilon}$ over the metric DS $(\Omega, \hat{\mathscr{F}}, \hat{\mathbb{P}}, \hat{\theta})$.
	
	\begin{lemma}\label{Bddns4}
		For each $\omega\in \Omega,$
		\begin{align*}
		\limsup_{t\to - \infty} \|\z(\omega)(t)\|^2_{\H}\  e^{\mu \lambda_1 t +\frac{216}{ \mu^3} \int_{t}^{0}\|\z(\zeta)\|^{4}_{\V}\d\zeta} = 0.
		\end{align*}
	\end{lemma}
	\begin{proof}
		For proof, see Lemma 4.1 \cite{KM1}.
	\end{proof}
	\begin{lemma}\label{Bddns5}
		For each $\omega\in \Omega,$
		\begin{align*}
		\int_{- \infty}^{0} \bigg\{ 1 +\|\z(t)\|^4_{\V}  + \|\A\z(t)\|^{r+1}_{\H}+   \|\z(t)\|^2_{\V}   \bigg\}e^{\mu \lambda_1 t +\frac{216}{\mu^3} \int_{t}^{0}\|\z(\zeta)\|^{4}_{\V}\d\zeta} \d t < \infty.
		\end{align*}
	\end{lemma}
	\begin{proof}
		For proof, see Lemma 4.2 \cite{KM1}.
	\end{proof}
	\begin{definition}\label{RA2}
		A function $\kappa: \Omega\to (0, \infty)$ belongs to class $\mathfrak{K}$ if and only if 
		\begin{align}
		\limsup_{t\to \infty} [\kappa(\theta_{-t}\omega)]^2 e^{-\mu \lambda_1 t +\frac{216}{\mu^3} \int_{-t}^{0}\|\z(\omega)(s)\|^{4}_{\V}\d s} = 0,
		\end{align}
		where $\lambda_1$ is the first eigenvalue of Stokes operator $\A.$
		
		We denote by $\mathfrak{DK}$ the class of all closed and bounded random sets $\mathrm{D}$ on $\H$ such that the radius function $\Omega\ni \omega \mapsto \kappa(\mathrm{D}(\omega)):= \sup\{\|x\|_{\H}:x\in \mathrm{D}(\omega)\}$ belongs to the class $\mathfrak{K}.$
		
		By Corollary \ref{Bddns1}, we infer that the constant functions belong to $\mathfrak{K}$. The class $\mathfrak{K}$ is closed with respect to sum, multiplication by a constant and if $\kappa \in \mathfrak{K}, 0\leq \bar{\kappa} \leq \kappa,$ then $\bar{\kappa}\in \mathfrak{K}.$
	\end{definition}
	\begin{proposition}\label{radius}
		Define functions $\kappa_{i}:\Omega\to (0, \infty), i= 1, 2, 3, 4, 5, 6,$ by the following formulae, for $\omega\in\Omega,$
		\begin{align*}
		[\kappa_1(\omega)]^2 &:= \|\z(\omega)(0)\|_{\H},\\
		[\kappa_2(\omega)]^2& := \sup_{s\leq 0} \|\z(\omega)(s)\|^2_{\H}\  e^{\mu \lambda_1 s +\frac{216}{ \mu^3} \int_{s}^{0}\|\z(\omega)(\zeta)\|^{4}_{\V}\d\zeta}, \\
		[\kappa_3(\omega)]^2 &:= \int_{- \infty}^{0} \|\z(\omega)(t)\|^{2}_{\V}\ e^{\mu \lambda_1 t +\frac{216}{ \mu^3} \int_{t}^{0}\|\z(\omega)(\zeta)\|^{4}_{\V}\d\zeta}\d t, \\
		[\kappa_4(\omega)]^2 &:= \int_{- \infty}^{0} \|\A\z(\omega)(t)\|^{r+1}_{\H} e^{\mu \lambda_1 t +\frac{216}{ \mu^3} \int_{t}^{0}\|\z(\omega)(\zeta)\|^{4}_{\V} \d\zeta} \d t,\\
		[\kappa_5(\omega)]^2 &:= \int_{- \infty}^{0} \|\z(\omega)(t)\|^4_{\V}\ e^{\mu \lambda_1 t +\frac{216}{ \mu^3} \int_{t}^{0}\|\z(\omega)(\zeta)\|^{4}_{\V}\d\zeta} \d t,\\
		[\kappa_6(\omega)]^2& := \int_{- \infty}^{0} e^{\mu \lambda_1 t +\frac{216}{ \mu^3} \int_{t}^{0}\|\z(\omega)(\zeta)\|^{4}_{\V}\d\zeta} \d t.
		\end{align*}
		Then all these functions belongs to class $\mathfrak{K}.$
	\end{proposition}
	\begin{proof}
		For proof, see Proposition 4.4 \cite{KM1}.
	\end{proof}

	\begin{theorem}\label{H_ab}
		Assume that $\f\in\H$ and Assumption \ref{A1} holds. Then there exists a family $\hat{B}_0=\{\mathrm{{\bf B}}_0(\omega):\omega\in \Omega\}$ of $\mathfrak{DK}$-random absorbing sets in $\H$ and a family $\hat{B}=\{\mathrm{{\bf B}}(\omega):\omega\in \Omega\}$ of $\mathfrak{DK}$-random absorbing sets in $\V$ corresponding to the RDS $\varphi_{\varepsilon}.$
	\end{theorem}
	\begin{proof}
		We divide the proof into the following steps.
		\vskip 0.2 cm
		\noindent\textbf{Step I:} \emph{Absorbing set in $\H$.} 
		Let $\mathrm{D}$ be a random set from the class $\mathfrak{DK}$. Let $\kappa_{\mathrm{D}}(\omega)$ be the radius of $\mathrm{D}(\omega)$, that is, $\kappa_{\mathrm{D}}(\omega):= \sup\{\|x\|_{\H} : x \in \mathrm{D}(\omega)\}, \omega\in \Omega.$
		
		Let $\omega\in \Omega$ be fixed. For given $s\leq 0$ and $\x\in \H$, let $\v(\cdot)$ be the solution of \eqref{cscbf} on time interval $[s, \infty)$ with the initial condition $\v_{\varepsilon}(s)= \x-\varepsilon\z(s).$
		Multiplying the first equation of \eqref{cscbf} by $\v_{\varepsilon}(\cdot)$ and then integrating the resulting equation over $\mathbb{T}^3$, we obtain	
		\begin{align}\label{H_ab1}
		\frac{1}{2}\frac{\d}{\d t}\|\v_{\varepsilon}(t)\|^2_{\H} = & -\mu\|\v_{\varepsilon}(t)\|^2_{\V} -b(\v_{\varepsilon}(t)+\varepsilon\z(t),\v_{\varepsilon}(t)+\varepsilon\z(t),\v_{\varepsilon}(t))\nonumber\\&-\beta\big\langle\mathcal{C}(\v_{\varepsilon}(t)+\varepsilon\z(t)),\v_{\varepsilon}(t)\big\rangle +\varepsilon\alpha(\z(t), \v_{\varepsilon}(t)) +(\f,\v_{\varepsilon}(t))\nonumber\\
		=&-\mu\|\v_{\varepsilon}(t)\|^2_{\V} -\varepsilon b(\v_{\varepsilon}(t),\z(t),\v_{\varepsilon}(t))-\varepsilon^2b(\z(t),\z(t),\v_{\varepsilon}(t))\nonumber\\&+\beta\big\langle\mathcal{C}(\v_{\varepsilon}(t)+\varepsilon\z(t)),\varepsilon\z(t)\big\rangle-\beta\big\langle\mathcal{C}(\v_{\varepsilon}(t)+\varepsilon\z(t)),\v_{\varepsilon}(t)+\varepsilon\z(t)\big\rangle \nonumber\\&+\varepsilon\alpha(\z(t), \v_{\varepsilon}(t)) +(\f,\v_{\varepsilon}(t)),
		\end{align}
		for a.e. $t\in[0,T].$ By using $0<\varepsilon\leq1$, H\"older's and Young's inequalities, \eqref{poin}, \eqref{lady} and Sobolev's inequality, we have
		\begin{align*}
		|\varepsilon b(\v_{\varepsilon},\z, \v_{\varepsilon})|&\leq \|\v_{\varepsilon}\|^2_{\L^{4}}\|\z\|_{\V}\leq2\|\v_{\varepsilon}\|^{\frac{1}{2}}_{\H}\|\v_{\varepsilon}\|^{\frac{3}{2}}_{\V}\|\z\|_{\V}\leq  \frac{\mu}{8} \|\v_{\varepsilon}\|^2_{\V} +\frac{108}{\mu^{3}}\|\z\|^4_{\V}\|\v_{\varepsilon}\|^2_{\H}, \\
		|\varepsilon^2b(\z,\z, \v_{\varepsilon})|&\leq \|\z\|^2_{\L^{4}}\|\v_{\varepsilon}\|_{\V}\leq2\|\z\|^{\frac{1}{2}}_{\H}\|\z\|^{\frac{3}{2}}_{\V}\|\v_{\varepsilon}\|_{\V}\leq  \frac{8}{\mu\lambda_1^{1/2}}\|\z\|^4_{\V} + \frac{\mu}{8} \|\v_{\varepsilon}\|^2_{\V},\\
		\beta\big\langle\mathcal{C}(\v_{\varepsilon}+\varepsilon\z),\v_{\varepsilon}+\varepsilon\z\big\rangle& =   \beta\|\v_{\varepsilon}+\varepsilon\z\|^{r+1}_{\L^{r+1}},\\
		|\beta\big\langle\mathcal{C}(\v_{\varepsilon}+\varepsilon\z),\varepsilon\z\big\rangle|&\leq \beta \|\v_{\varepsilon}+\varepsilon\z\|^{r}_{\L^{r+1}} \|\z\|_{\L^{r+1}}\leq \frac{\beta}{2} \|\v_{\varepsilon}+\varepsilon\z\|^{r+1}_{\L^{r+1}} + \frac{\beta(2r)^r}{(r+1)^{r+1}}\|\z\|^{r+1}_{\L^{r+1}}\\
		&\leq \frac{\beta}{2} \|\v_{\varepsilon}+\varepsilon\z\|^{r+1}_{\L^{r+1}} + C\|\A\z\|^{r+1}_{\H},\\
		|\varepsilon\alpha( \z, \v_{\varepsilon}) | &\leq  \alpha\|\z\|_{\H} \|\v_{\varepsilon}\|_{\H}\leq  \frac{\alpha}{\lambda_1}\|\z\|_{\V} \|\v_{\varepsilon}\|_{\V}\leq  \frac{\mu}{8} \|\v_{\varepsilon}\|^2_{\V} + \frac{4\alpha^2}{\mu\lambda_1^2} \|\z\|^2_{\V},\\
		|(\f, \v_{\varepsilon}) |& \leq  \|\f\|_{\H} \|\v_{\varepsilon}\|_{\H}\leq  \frac{1}{\lambda_1^{1/2}}\|\f\|_{\H} \|\v_{\varepsilon}\|_{\V}\leq   \frac{\mu}{8} \|\v_{\varepsilon}\|^2_{\V} + \frac{4}{\mu\lambda_1} \|\f\|^2_{\H}.
		\end{align*}
		Combining the above estimates and substituting it in \eqref{H_ab1}, we get
		\begin{align}\label{H_ab2}
		&\frac{\d}{\d t} \|\v_{\varepsilon}(t)\|^2_{\H}  + \mu\lambda_1 \|\v_{\varepsilon}(t)\|^2_{\H}  \nonumber\\&\leq  \frac{216}{\mu^3}\|\z(t)\|^4_{\V}\|\v_{\varepsilon}(t)\|^2_{\H} +\frac{16}{\mu\lambda_1^{1/2}}\|\z(t)\|^4_{\V}+ C\|\A\z(t)\|^{r+1}_{\H}+  \frac{8\alpha^2}{\mu\lambda_1^2} \|\z(t)\|^2_{\V} + \frac{8}{\mu\lambda_1} \|\f\|^2_{\H},
		\end{align}
		for a.e. $t\in[0,T].$ We infer from the classical Gronwall inequality that 
		\begin{align}\label{H_ab3}
		\|\v_{\varepsilon}(0)\|^2_{\H} 
		&\leq 2 \|\x\|^2_{\H} e^{\mu\lambda_1 s + \frac{216}{\mu^3}\int_{s}^{0}\|\z(\zeta)\|^4_{\V}\d\zeta} +2\|\z(s)\|^2_{\H} e^{\mu\lambda_1 s + \frac{216}{\mu^3}\int_{s}^{0}\|\z(\zeta)\|^4_{\V}\d\zeta} +\int_{s}^{0}\bigg\{\frac{16}{\mu\lambda_1^{1/2}}\|\z(t)\|^4_{\V} \nonumber \\& \quad + C\|\A\z(t)\|^{r+1}_{\H}+  \frac{8\alpha^2}{\mu\lambda_1^2} \|\z(t)\|^2_{\V}  + \frac{8}{\mu\lambda_1} \|\f\|^2_{\H}\biggr\}e^{\mu\lambda_1t  + \frac{216}{\mu^3}\int_{t}^{0}\|\z(\zeta)\|^4_{\V}\d\zeta}\d t,
		\end{align}
		for all $t\in[0,T]$, since $\|\v_{\e}(s)\|_{\H}\leq\|\x\|_{\H}+\|\z(s)\|_{\H}$. For $\omega\in \Omega,$ let us set
		\begin{align}
		[\kappa_{11}(\omega)]^2 =  &\ \ 2 +  2\sup_{s\leq 0}\bigg\{ \|\z(s)\|^2_{\H}\  e^{\mu \lambda_1 s +\frac{216}{ \mu^3} \int_{s}^{0}\|\z(\zeta)\|^{4}_{\V}\d\zeta}\bigg\} + \int_{- \infty}^{0} \bigg\{\frac{16}{\mu\lambda_1^{1/2}}\|\z(t)\|^4_{\V}  \nonumber\\& \quad + C\|\A\z(t)\|^{r+1}_{\H}+  \frac{8\alpha^2}{\mu\lambda_1^2} \|\z(t)\|^2_{\V}  + \frac{8}{\mu\lambda_1} \|\f\|^2_{\H}\bigg\}e^{\mu \lambda_1 t +\frac{216}{ \mu^3} \int_{t}^{0}\|\z(\zeta)\|^{4}_{\V}\d\zeta} \d t,\\	
		\kappa_{12}(\omega)=& \ \  \|\z(\omega)(0)\|_{\H}.
		\end{align}
		By Lemma \ref{Bddns5} and Proposition \ref{radius}, we infer that both $\kappa_{11}$ and $\kappa_{12}$ belongs to class $\mathfrak{K}$ and also that $\kappa_{13}:=\kappa_{11}+\kappa_{12}$ belongs to class $\mathfrak{K}$ as well. Therefore the random set $\textbf{B}_0(\cdot)$  defined by $$\textbf{B}_0(\omega) := \{\u\in\H: \|\u\|_{\H}\leq \kappa_{13}(\omega)\}$$ belongs to the family $\mathfrak{DK}.$
		
		Next, we  show that $\textbf{B}_0$ absorbs $\mathrm{D}$. Let $\omega\in\Omega$ be fixed. Since $\kappa_{\mathrm{D}}(\omega)\in \mathfrak{K}$, there exists $t_{\mathrm{D}}(\omega)\geq 0,$ such that 
		\begin{align*}
		[\kappa_{\mathrm{D}}(\theta_{-t}\omega)]^2 e^{-\mu \lambda_1 t +\frac{216}{ \mu^3} \int_{-t}^{0}\|\z(\omega)(s)\|^{4}_{\V}\d s} \leq 1, \  \text{for}\  t\geq t_{\mathrm{D}}(\omega). 
		\end{align*}
		Thus, if $\x\in \mathrm{D}(\theta_{-t}\omega)$ and $s\leq- t_{\mathrm{D}}(\omega),$ then by \eqref{H_ab3}, we have $$\|\v_{\varepsilon}(0,\omega; s, \x-\varepsilon\z(s))\|_{\H}\leq \kappa_{11}(\omega).$$ Hence, we deduce that $$\|\u_{\varepsilon}(0,\omega; s, \x)\|_{\H} \leq \|\v_{\varepsilon}(0,\omega; s, \x-\z(s))\|_{\H} + \|\z(\omega)(0)\|_{\H}\leq \kappa_{13}(\omega).$$ This implies that $\u_{\varepsilon}(0,\omega; s, \x) \in \textbf{B}_0(\omega)$, for all $s\leq -t_{\mathrm{D}}(\omega),$ and hence it follows that $\textbf{B}_0$ absorbs $\mathrm{D}$.

		Moreover, integrating \eqref{H_ab2} over $(-1,0)$, we find for any $\omega\in \Omega$, there exists an $\kappa_{14}(\omega)\geq 0$ such that
		\begin{align}\label{H_ab4}
		\int_{-1}^{0} \bigg[\|\v_{\varepsilon}(t)\|^2_{\V} + \|\v_{\varepsilon}(t)+ \varepsilon\z(t)\|^{r+1}_{\L^{r+1}} \bigg]\d t \leq \kappa_{14}(\omega),
		\end{align}
		for all $s\leq -t_{\mathrm{D}}(\omega).$
		
		\vskip 0.2 cm
		\noindent\textbf{Step II:} \emph{Absorbing set in $\V$.} 
		Multiplying the first equation in \eqref{cscbf} by $\A\v_{\varepsilon}(\cdot)$ and then integrating the resulting equation over $\mathbb{T}^3$, we obtain
		\begin{align}\label{H_ab5}
		\frac{1}{2}\frac{\d}{\d t}\|\v_{\varepsilon}(t)\|^2_{\V} = & -\mu\|\A\v_{\varepsilon}(t)\|^2_{\V} -b(\v_{\varepsilon}(t)+\varepsilon\z(t),\v_{\varepsilon}(t)+\varepsilon\z(t),\A\v_{\varepsilon}(t))\nonumber\\&-\beta(\mathcal{C}(\v_{\varepsilon}(t)+\varepsilon\z(t)),\A\v_{\varepsilon}(t)) +\alpha(\varepsilon\z(t), \A\v_{\varepsilon}(t)) +(\f,\A\v_{\varepsilon}(t))\nonumber\\
		=&-\mu\|\A\v_{\varepsilon}(t)\|^2_{\V} -b(\u_{\varepsilon}(t),\u_{\varepsilon}(t),\A\v_{\varepsilon}(t))-\beta(\mathcal{C}(\u_{\varepsilon}(t)),\A\u_{\varepsilon}(t))\nonumber\\&+\beta(\mathcal{C}(\u_{\varepsilon}(t)),\varepsilon\A\z(t)) +\alpha(\varepsilon\z(t), \A\v_{\varepsilon}(t)) +(\f,\A\v_{\varepsilon}(t)),
		\end{align}
		where $\u_{\varepsilon}(t)=\v_{\varepsilon}(t)+\varepsilon\z(t),$ for a.e. $t\in[0,T].$ We consider the cases $r>3$ and $r=3$ separately. 
		\vskip 0.2 cm 
		\noindent
		\textbf{Case I:} $r>3$.
		From \eqref{3}, we have
		\begin{align}
		(\mathcal{C}(\u_{\varepsilon}),\A\u_{\varepsilon})&=\||\nabla\u_{\varepsilon}||\u_{\varepsilon}|^{\frac{r-1}{2}}\|^2_{\H}+4\left[\frac{r-1}{(r+1)^2}\right]\|\nabla|\u_{\varepsilon}|^{\frac{r+1}{2}}\|^2_{\H},\label{V_ab1}\\
		|(\mathcal{C}(\u_{\varepsilon}),\varepsilon\A\z)|&\leq\|\u_{\varepsilon}\|^r_{\L^{2r}}\|\varepsilon\A\z\|_{\H}\leq C\|\u_\varepsilon\|^r_{\L^{3(r+1)}}\|\varepsilon\A\z\|_{\H}\nonumber\\&=C\||\u_{\varepsilon}|^{\frac{r+1}{2}}\|^{\frac{2r}{r+1}}_{\L^6}\|\varepsilon\A\z\|_{\H}\leq C\|\nabla|\u_{\varepsilon}|^{\frac{r+1}{2}}\|^{\frac{2r}{r+1}}_{\H}\|\varepsilon\A\z\|_{\H}\nonumber\\&\leq 2\left[\frac{r-1}{(r+1)^2}\right]\|\nabla|\u_{\varepsilon}|^{\frac{r+1}{2}}\|^2_{\H} + \varepsilon^{r+1}C\|\A\z\|^{r+1}_{\H},\label{V_ab2}\\
		\alpha|(\varepsilon\z,\A\v_{\varepsilon})|&\leq \alpha \|\varepsilon\z\|_{\H}\|\A\v_{\varepsilon}\|_{\H}\leq\frac{\mu}{6}\|\A\v_{\varepsilon}\|^2_{\H} + \varepsilon^2C\|\z\|^2_{\H},\label{V_ab3}\\
		|(\f, \A\v_{\varepsilon})|&\leq \|\f\|_{\H}\|\A\v_{\varepsilon}\|_{\H}\leq\frac{\mu}{6}\|\A\v_{\varepsilon}\|^2_{\H} + C\|\f\|^2_{\H}.\label{V_ab4}
		\end{align}
		We estimate $|b(\u_{\varepsilon},\u_{\varepsilon},\A\v_{\varepsilon})|$ using H\"older's, and Young's inequalities as 
		\begin{align}\label{V_ab5}
		|b(\u_{\varepsilon},\u_{\varepsilon},\A\v_{\varepsilon})|&\leq\||\u_{\varepsilon}||\nabla\u_{\varepsilon}|\|_{\H}\|\A\v_{\varepsilon}\|_{\H}\leq\frac{\mu}{6}\|\A\v_{\varepsilon}\|_{\H}^2+\frac{3}{2\mu }\||\u_{\varepsilon}||\nabla\u_{\varepsilon}|\|_{\H}^2. 
		\end{align}
		Let us  estimate the final term from \eqref{V_ab5} using H\"older's and Young's inequalities as (similarly as in \eqref{2.29})
		\begin{align}\label{V_ab6}
		&	\int_{\mathbb{T}^3}|\u_{\varepsilon}(x)|^2|\nabla\u_{\varepsilon}(x)|^2\d x\nonumber\\&\leq{\frac{\beta\mu}{3} }\left(\int_{\mathbb{T}^3}|\u_{\varepsilon}(x)|^{r-1}|\nabla\u_{\varepsilon}(x)|^2\d x\right)+\frac{r-3}{r-1}\left[\frac{6}{\beta\mu (r-1)}\right]^{\frac{2}{r-3}}\left(\int_{\mathbb{T}^3}|\nabla\u_{\varepsilon}(x)|^2\d x\right).
		\end{align}
		Making use of the estimate \eqref{V_ab6} in \eqref{V_ab5}, we find
		\begin{align}\label{V_ab7}
		|b(\u_{\varepsilon},\u_{\varepsilon},\A\v_{\varepsilon})|&\leq\frac{\mu}{6}\|\A\v_{\varepsilon}\|_{\H}^2+\frac{\beta}{2}\||\nabla\u_{\varepsilon}||\u_{\varepsilon}|^{\frac{r-1}{2}}\|^2_{\H}+\frac{3(r-3)}{2\mu(r-1)}\left[\frac{6}{\beta\mu (r-1)}\right]^{\frac{2}{r-3}}\|\u_{\varepsilon}\|^2_{\V}\nonumber\\&\leq\frac{\mu}{6}\|\A\v_{\varepsilon}\|_{\H}^2+\frac{\beta}{2}\||\nabla\u_{\varepsilon}||\u_{\varepsilon}|^{\frac{r-1}{2}}\|^2_{\H} +C\|\v_{\varepsilon}\|^2_{\V}+\varepsilon^2C\|\z\|^2_{\V}.
		\end{align}
		For any $\omega\in\Omega$,	we infer from the inequalities \eqref{V_ab1}-\eqref{V_ab7} that 
		\begin{align}\label{V_ab8}
		&\frac{\d}{\d t}\|\v_{\varepsilon}(t)\|^2_{\V} +\mu\|\A\v_{\varepsilon}(t)\|^2_{\H} +\beta\||\nabla\u_{\varepsilon}(t)||\u_{\varepsilon}(t)|^{\frac{r-1}{2}}\|^2_{\H}+4\beta\left[\frac{r-1}{(r+1)^2}\right]\|\nabla|\u_{\varepsilon}(t)|^{\frac{r+1}{2}}\|^2_{\H} \nonumber\\
		&\leq C\|\v_{\varepsilon}(t)\|^2_{\V}+\varepsilon^2C\|\z(t)\|^2_{\V} + \varepsilon^{r+1}C\|\A\z(t)\|^{r+1}_{\H} +\varepsilon^2C\|\z(t)\|^2_{\H}+C\|\f\|^2_{\H},
		\end{align}
		for a.e. $t\in[0,T].$ From the uniform Gronwall lemma  (Lemma 1.1, \cite{R.Temam}) and  \eqref{H_ab4}, we deduce that for any $\omega\in\Omega,$ there exists $\kappa_{15}(\omega)\geq0$ and $\kappa_{16}(\omega)\geq0$ such that 
		\begin{align}
		\|\v_{\varepsilon}(0,\omega; s, \x-\varepsilon\z(s))\|_{\V}\leq \kappa_{15}(\omega)
		\end{align}
		and
		\begin{align}
		&	\mu\int_{-1}^{0}\|\A\v_{\varepsilon}(s)\|^2_{\H} \d s+\beta\int_{-1}^{0}\||\nabla\u_{\varepsilon}(t)||\u_{\varepsilon}(t)|^{\frac{r-1}{2}}\|^2_{\H}\d s+4\beta\left[\frac{r-1}{(r+1)^2}\right]\int_{-1}^{0}\|\nabla|\u_{\varepsilon}(t)|^{\frac{r+1}{2}}\|^2_{\H}\d s\nonumber\\& \leq \kappa_{16}(\omega),
		\end{align}
		for any $s\leq-(t_{\mathrm{D}}(\omega)+1),$ which competes the proof for the case $r>3.$
		\vskip 0.2 cm
		\noindent
		\textbf{Case II:}  $r=3$ and $2\beta\mu\geq1$.
		From \eqref{3}, we have
		\begin{align}
		(\mathcal{C}(\u_{\varepsilon}),\A\u_{\varepsilon})&=\||\nabla\u_{\varepsilon}||\u_{\varepsilon}|\|^2_{\H}+\frac{1}{2}\|\nabla|\u_{\varepsilon}|^2\|^2_{\H},\label{V_ab9}\\
		|(\mathcal{C}(\u_{\varepsilon}),\varepsilon\A\z)|&\leq\|\u_{\varepsilon}\|^3_{\L^{6}}\|\varepsilon\A\z\|_{\H}\leq C\|\u_{\varepsilon}\|^3_{\L^{12}}\|\varepsilon\A\z\|_{\H}=C\||\u_{\varepsilon}|^{2}\|^{\frac{3}{2}}_{\L^6}\|\varepsilon\A\z\|_{\H}\nonumber\\
		&\leq C\|\nabla|\u_{\varepsilon}|^{2}\|^{\frac{3}{2}}_{\H}\|\varepsilon\A\z\|_{\H}\leq \frac{1}{4}\|\nabla|\u_{\varepsilon}|^{2}\|^2_{\H} + \varepsilon^4C\|\A\z\|^{4}_{\H},\label{V_ab10}\\
		\alpha|(\varepsilon\z,\A\v_{\varepsilon})|&\leq \alpha \|\varepsilon\z\|_{\H}\|\A\v_{\varepsilon}\|_{\H}	\leq\frac{\mu}{4}\|\A\v_{\varepsilon}\|^2_{\H} + \varepsilon^2C\|\z\|^2_{\H},\label{V_ab11}\\
		|(\f, \A\v_{\varepsilon})|&\leq \|\f\|_{\H}\|\A\v_{\varepsilon}\|_{\H}\leq\frac{\mu}{4}\|\A\v_{\varepsilon}\|^2_{\H} + C\|\f\|^2_{\H}.\label{V_ab12}
		\end{align}
		We estimate $|b(\u_{\varepsilon},\u_{\varepsilon},\A\v_{\varepsilon})|$ using H\"older's and Young's inequalities as 
		\begin{align}\label{V_ab13}
		|b(\u_{\varepsilon},\u_{\varepsilon},\A\v_{\varepsilon})|&\leq\||\u_{\varepsilon}||\nabla\u_{\varepsilon}|\|_{\H}\|\A\v_{\varepsilon}\|_{\H}\leq\frac{1}{4\beta}\|\A\v_{\varepsilon}\|_{\H}^2+\beta\||\u_{\varepsilon}||\nabla\u_{\e}|\|_{\H}^2. 
		\end{align}
		For any $\omega\in\Omega$,	we infer from the inequalities \eqref{V_ab9}-\eqref{V_ab13} and \eqref{H_ab5} that 
		\begin{align}\label{V_ab14}
		\frac{\d}{\d t}\|\v_{\varepsilon}(t)\|^2_{\V} +\left(\mu-\frac{1}{2\beta}\right)\|\A\v_{\varepsilon}(t)\|^2_{\H} +\beta\|\nabla|\u_{\varepsilon}(t)|^{2}\|^2_{\H} \leq 
		\varepsilon^4C\|\A\z(t)\|^{4}_{\H} +\varepsilon^2C\|\z(t)\|^2_{\H}+C\|\f\|^2_{\H},
		\end{align}
		for a.e. $t\in[0,T].$ From the uniform Gronwall lemma and \eqref{H_ab4}, we deduce that for any $\omega\in\Omega,$ there exists $\kappa_{17}(\omega)\geq0$ and $\kappa_{18}(\omega)\geq0$ such that 
		\begin{align}
		\|\v_{\varepsilon}(0,\omega; s, \x-\varepsilon\z(s))\|_{\V}\leq \kappa_{17}(\omega)
		\end{align}
		and
		\begin{align}
		\left(\mu-\frac{1}{2\beta}\right)\int_{-1}^{0}\|\A\v_{\varepsilon}(s)\|^2_{\H} \d s+\frac{\beta}{2}\int_{-1}^{0}\|\nabla|\u_{\varepsilon}(t)|^{2}\|^2_{\H}\d s \leq \kappa_{18}(\omega),
		\end{align}
		for any $s\leq-(t_{\mathrm{D}}(\omega)+1),$ which competes the proof for $r=3$ and  $2\beta\mu\geq1$.
	\end{proof}
	
	Thanks to the compactness of $\V$ in $\H$, from Theorem \ref{H_ab} and the abstract theory of random attractors (Theorem 3.11, \cite{CF}), we immediately conclude the following result.
	
	\begin{theorem}\label{Main_Theoem}
		Assume that $\f\in\H$ and Assumption \ref{assump} holds. Then, the cocycle $\varphi_{\varepsilon}$ corresponding to the 3D stochastic CBF equations with small additive noise \eqref{S-CBF} for $r\geq3$ ($r>3$, for any $\mu$ and $\beta$,  $r=3$  for $2\beta\mu\geq1$) has a random attractor $\mathcal{A}_{\varepsilon} = \{\textbf{A}_{\varepsilon}(\omega): \omega\in \Omega\}$ in $\H$. 
	\end{theorem}

	\section{Upper semi-continuity of $\mathfrak{DK}$-random attractors in $\H$}\label{sec5}\setcounter{equation}{0}
	The aim of this section is to prove the upper semicontinuity of random attractors in $\H$. The existence of random attractors for the stochastic system \eqref{S-CBF} in $\H$ is proved in Theorem \ref{Main_Theoem} and the existence of global attractors for the deterministic system \eqref{D-CBF}  in $\H$ is proved in Theorem \ref{Det-Attractor}. Upper semicontinuity results for several infinite dimensional  stochastic models is obtained in \cite{CLR, KM1,FLYB} etc. In particular, authors in \cite{KM1}, obtained the upper semicontinuity results for 2D SCBF equations for $r\in[1,3]$. Using similar techniques used in the work \cite{CLR}, we state and prove the following result on the upper semicontinuity of the random attractors: 
	\begin{theorem}\label{USC}
		For $r\geq3$ ($r>3$, for any $\mu$ and $\beta$,  and $r=3$  for $2\beta\mu\geq1$), assume that $\f\in\H$ and Assumption \ref{assump} is satisfied. Also, assume that the deterministic system \eqref{D-CBF} has a global attractor $\mathcal{A}$ and its small random perturbed dynamical system \eqref{S-CBF} possesses a $\mathfrak{DK}$-random attractor $\mathcal{A}_\varepsilon= \{\mathbf{A}_\varepsilon(\omega):\omega\in \Omega\},$ for any $\varepsilon\in (0, 1].$ Let the following conditions hold:
		\begin{itemize}
			\item [$(K_1)$] For each $t_0\geq0$ and for $\hat{\mathbb{P}}$-a.e. $\omega\in\Omega,$ $$\lim_{\varepsilon\to0^+} d(\varphi_\varepsilon(t_0, \theta_{-t_0}\omega)\x,\S(t)\x)=0,$$
			uniformly on bounded sets of $\H$, where $\varphi_\varepsilon$ is a RDS and $\S(t)$ is a semigroup generated by \eqref{S-CBF} and \eqref{D-CBF}, respectively with the same initial condition $\x\in\H$.  
			\item[$(K_2)$] There exists a compact set $K\subset \H$ such that 
			$$\lim_{\varepsilon\to 0^+} d(\mathbf{A}_{\varepsilon}(\omega), K)=0,$$for $\hat{\mathbb{P}}$-a.e. $\omega\in\Omega.$
		\end{itemize}
		Then $\mathcal{A}_\varepsilon$ and $\mathcal{A}$ have the property of upper semi-continuity, that is,\begin{align}\label{51}\lim_{\varepsilon\to 0^+} d(\mathcal{A}_{\varepsilon}(\omega), \mathcal{A})=0,\end{align}for $\hat{\mathbb{P}}$-a.e. $\omega\in\Omega.$

		Furthermore, if for $\varepsilon_0\in (0,1]$ we have that for $\hat{\mathbb{P}}$-a.e. $\omega\in\Omega$ and all $t_0>0$ \begin{align}\label{K_3}
		\varphi_{\varepsilon}(t_0, \theta_{-t_0}\omega)\x \to \varphi_{\varepsilon_0}(t_0, \theta_{-t_0}\omega)\x\ \text{ as }\ \varepsilon\to\varepsilon_0,
		\end{align} uniformly on bounded sets of $\H$, then the convergence \eqref{K_3} is upper semicontinuous in $\varepsilon$, that is, \begin{align}\label{52}
		\lim_{\varepsilon\to\varepsilon_0} d(\mathcal{A}_{\varepsilon}(\omega),\mathcal{A}_{\varepsilon_0}(\omega)) = 0,
		\end{align}for $\hat{\mathbb{P}}$-a.e. $\omega\in\Omega.$
		
	\end{theorem}
	\begin{proof}
		To prove the property of upper semicontinuity for our system, we only need to verify the conditions $(K_1)$ and $(K_2)$.
		\vskip 0.2cm
		\noindent\textbf{Step I.} \emph{Verification of $(K_2)$:} Let us introduce $$\v_\varepsilon(t,\omega)=\u_\varepsilon(t, \omega)-\varepsilon\z(t, \omega),$$ where $\u_\varepsilon(t, \omega)$ and $\z(t,\omega)$ are the unique solutions of \eqref{S-CBF} and \eqref{OUPe}, respectively. Also from \eqref{O-U_conti}, we have $\z\in\mathrm{L}^{\infty}_{\mathrm{loc}}([t_0,\infty);\V)\cap\mathrm{L}^{r+1}_{\mathrm{loc}}([t_0,\infty);\D(\A))$. Clearly, $\v_\varepsilon(\cdot)$ satisfies
		\begin{equation}\label{cscbf_s}
		\begin{aligned}
		\frac{\d\v_\varepsilon(t)}{\d t} &= -\mu \A\v_\varepsilon(t)- \B(\v_\varepsilon (t)+ \varepsilon\z(t)) - \beta \mathcal{C}(\v_\varepsilon (t)+ \varepsilon\z(t)) + \varepsilon\alpha \z(t) + \f, 
		\end{aligned}
		\end{equation}
		in $\V'+\wi\L^{\frac{r+1}{r}}$.	From Theorem \ref{H_ab}, we observe that there exists $\hat{\kappa}_\varepsilon(\omega)\in\mathfrak{K}$-class such that $$\|\v_\varepsilon(0)\|_{\V}\leq\hat{\kappa}_\varepsilon(\omega).$$ If we call $K_\varepsilon(\omega),$ the ball in $\V$ of radius $\hat{\kappa}_\varepsilon(\omega)+\varepsilon\|\z(0)\|_{\V}$, we have a compact (since $\V\hookrightarrow\H$ is compact) $\mathfrak{DK}$-absorbing set in $\H$ for $\varphi_\varepsilon$. Furthermore, there exists a $\hat{\kappa}_d$ independent of $\omega\in \Omega$ such that $$\lim_{\varepsilon\to0^+} \hat{\kappa}_\varepsilon(\omega)\leq\hat{\kappa}_d,$$which easily verifies  Lemma 1, \cite{CLR} and hence $(K_2)$ follows.
		\vskip 0.2cm
		\noindent
		\textbf{Step II.} \emph{Verification of $(K_1)$:} In order to verify the assertion $(K_1)$, it is enough to prove that the solution $\varphi_\varepsilon(t,\omega)\x$ of system \eqref{S-CBF} $\hat{\mathbb{P}}$-a.s. converges to the solution $\S(t)\x$ of the unperturbed system \eqref{D-CBF} in $\H$ as $\varepsilon\to 0^+$ uniformly on bounded sets of initial conditions. That is, for $\hat{\mathbb{P}}$-a.e $\omega\in\Omega$, any $t_0\geq0$ and any bounded subset $\G\subset\H$, we have 
		\begin{align}\label{usc}
		\lim_{\varepsilon\to 0^+}\sup_{\x\in \G} \|\varphi_\varepsilon(t_0,\theta_{-t_0}\omega)\x-\S(t_0)\x\|_{\H}=0.
		\end{align}
		For any $\x\in \G,$ let $\u_\varepsilon(t)=\varphi_{\varepsilon}(t+t_0,\theta_{-t_0})\x$ and $\u(t)=\S(t+t_0)\x$ respectively, be the unique weak solutions of the systems \eqref{S-CBF} and \eqref{D-CBF} with the initial condition $\x$ at $t=-t_0$. Also, for $T\geq 0$, let $$\upeta_\varepsilon(t)= \u_\varepsilon(t)-\u(t), \ t\in[-t_0,T].$$ Clearly $\upeta_\varepsilon(\cdot)$ satisfies
		\begin{equation}\label{cSCBF_y}
		\left\{
		\begin{aligned}
		\d\upeta_\varepsilon+\{\mu \A\upeta_\varepsilon+\B(\upeta_\varepsilon+\u)-\B(\u)+\beta \mathcal{C}(\upeta_\varepsilon+\u)-\beta\mathcal{C}(\u)\}\d t&= \varepsilon\d\text{W}(t),  \\ 
		\upeta_\varepsilon(-t_0)&=\mathbf{0},
		\end{aligned}
		\right.
		\end{equation}
		in $\V'+\wi\L^{\frac{r+1}{r}},$  for a.e. $ t\in[-t_0,T]$.	Let us introduce $\uprho_\varepsilon(\cdot)=\upeta_\varepsilon(\cdot)-\varepsilon\z(\cdot)$, where $\z(\cdot)$ is the solution of \eqref{OUPe}, then $\uprho_\varepsilon(\cdot)$ satisfies the following equation in $\V'$:
		\begin{equation}\label{cscbf_eta1}
		\left\{
		\begin{aligned}
		\frac{\d\uprho_\varepsilon}{\d t} &= -\mu \A\uprho_\varepsilon - \B(\uprho_\varepsilon + \varepsilon\z+\u)+\B(\u) - \beta \mathcal{C}(\uprho_\varepsilon + \varepsilon\z +\u)+\beta\mathcal{C}(\u) +\varepsilon \alpha \z, \\
		\uprho_\varepsilon(-t_0)&= -\varepsilon \z(-t_0),
		\end{aligned}
		\right.
		\end{equation}
		in $\V'+\wi\L^{\frac{r+1}{r}},$  for a.e. $ t\in[-t_0,T]$.	Taking the inner product of the first equation of \eqref{cscbf_eta1} with $\uprho_\varepsilon(\cdot)$, we obtain 
		\begin{align}\label{usc1}
		\frac{1}{2} \frac{\d}{\d t}\|\uprho_\varepsilon(t)\|^2_{\H}  &=-  \mu \|\uprho_\varepsilon(t)\|^2_{\V} - \big\langle\B(\uprho_\varepsilon(t) + \varepsilon\z(t)+\u(t))- \B(\u(t)), \uprho_\varepsilon(t)\big\rangle\nonumber\\&\ \ \ - \beta\big\langle\mathcal{C}(\uprho_\varepsilon(t) + \varepsilon\z(t)+\u(t))-\mathcal{C}(\u(t)), \uprho_\varepsilon(t)\big\rangle + \alpha (\varepsilon \z(t), \uprho_\varepsilon(t)), \nonumber\\&=-  \mu \|\uprho_\varepsilon(t)\|^2_{\V}+\big\langle\B(\uprho_\varepsilon(t) + \varepsilon\z(t)+\u(t)), \varepsilon \z(t)\big\rangle-\big\langle\B(\u(t)), \varepsilon \z(t)\big\rangle\nonumber\\&\ \ \ -\big\langle\B(\uprho_\varepsilon(t) + \varepsilon\z(t)+\u(t))- \B(\u(t)), (\uprho_\varepsilon(t) + \varepsilon\z(t)+\u(t))-(\u(t))\big\rangle\nonumber\\&\ \ \ +\beta\big\langle\mathcal{C}(\uprho_\varepsilon(t) + \varepsilon\z(t)+\u(t))-\mathcal{C}(\u(t)), \varepsilon \z(t)\big\rangle\nonumber\\&\ \ \ - \beta\big\langle\mathcal{C}(\uprho_\varepsilon(t) + \varepsilon\z(t)+\u(t))-\mathcal{C}(\u(t)),(\uprho_\varepsilon(t) + \varepsilon\z(t)+\u(t))-(\u(t))\big\rangle\nonumber\\&\quad + \alpha (\varepsilon\z(t), \uprho_\varepsilon(t)),
		\end{align}	
		for a.e. $\in[-t_0,T]$. 
		\vskip 0.2 cm
		\noindent
		\textbf{Case I:} $r>3$.
		Using  \eqref{212}, we have
		\begin{align}
		|\big\langle\B(\uprho_\varepsilon + \varepsilon\z+\u), \varepsilon\z\big\rangle|\leq&\|\uprho_\varepsilon + \varepsilon\z+\u\|_{\widetilde{\L}^{r+1}}^{\frac{r+1}{r-1}}\|\uprho_\varepsilon + \varepsilon\z+\u\|_{\H}^{\frac{r-3}{r-1}}\|\varepsilon\z\|_{\V}\nonumber\\\leq&C\|\uprho_\varepsilon + \varepsilon\z+\u\|_{\widetilde{\L}^{r+1}}^{\frac{r+1}{r-1}}\|\uprho_\varepsilon + \varepsilon\z+\u\|_{\V}^{\frac{r-3}{r-1}}\|\varepsilon\z\|_{\V},\label{usc2}\\
		|\big\langle\B(\u), \varepsilon\z\big\rangle|\leq&\|\u\|_{\widetilde{\L}^{r+1}}^{\frac{r+1}{r-1}}\|\u\|_{\H}^{\frac{r-3}{r-1}}\|\varepsilon\z\|_{\V}\leq C\|\u\|_{\widetilde{\L}^{r+1}}^{\frac{r+1}{r-1}}\|\u\|_{\V}^{\frac{r-3}{r-1}}\|\varepsilon\z\|_{\V}.\label{usc3}
		\end{align}
		Using \eqref{441} and Young's inequality (as in equation \eqref{2.16}), we deduce that 
		\begin{align}\label{usc4}
		&|\langle\B(\uprho_\varepsilon + \varepsilon\z+\u)-\B(\u),(\uprho_\varepsilon + \varepsilon\z+\u)-(\u)\rangle|\nonumber\\&\leq\frac{\mu }{4}\|(\uprho_\varepsilon + \varepsilon\z+\u)-(\u)\|_{\V}^2+\frac{\beta}{4}\||\u|^{\frac{r-1}{2}}((\uprho_\varepsilon + \varepsilon\z+\u)-(\u))\|_{\H}^2+\eta_4\|(\uprho_\varepsilon + \varepsilon\z+\u)-(\u)\|_{\H}^2\nonumber\\&\leq\frac{\mu}{2}\|\uprho_\varepsilon\|_{\V}^2+C\|\varepsilon\z\|_{\V}^2+\frac{\beta}{4}\||\u|^{\frac{r-1}{2}}((\uprho_\varepsilon + \varepsilon\z+\u)-(\u))\|_{\H}^2+2\eta_4\|\uprho_\varepsilon\|_{\H}^2,
		\end{align}
		where $\eta_4=\frac{r-3}{\mu(r-1)}\left(\frac{8}{\beta\mu (r-1)}\right)^{\frac{2}{r-3}}.$
		Now using Taylor's formula, we obtain 
		\begin{align}\label{usc5}
		&\beta\big|\big\langle\mathcal{C}(\uprho_\varepsilon + \varepsilon\z+\u)-\mathcal{C}(\u), \varepsilon \z\big\rangle\big|\nonumber\\&= \beta\bigg|\bigg\langle \int_{0}^{1} \big[\mathcal{C}'(\theta(\uprho_\varepsilon + \varepsilon\z+\u) + (1-\theta)\u)((\uprho_\varepsilon + \varepsilon\z+\u)-(\u))\big] \d\theta ,\varepsilon\z\bigg\rangle\bigg|\nonumber\\& \leq r\beta2^{r-2} \||(\uprho_\varepsilon + \varepsilon\z+\u)|^{\frac{r-1}{2}}(\uprho_\varepsilon + \varepsilon\z)\|_{\H} \|(\uprho_\varepsilon + \varepsilon\z+\u)\|_{\widetilde{\L}^{r+1}}^{\frac{r-1}{2}} \|\varepsilon\z\|_{\widetilde{\L}^{r+1}}\nonumber\\&\quad+r\beta2^{r-2} \||\u|^{\frac{r-1}{2}}(\uprho_\varepsilon + \varepsilon\z)\|_{\H} \|\u\|_{\widetilde{\L}^{r+1}}^{\frac{r-1}{2}} \|\varepsilon\z\|_{\widetilde{\L}^{r+1}}\nonumber\\&\leq\frac{\beta}{2} \||(\uprho_\varepsilon + \varepsilon\z+\u)|^{\frac{r-1}{2}}(\uprho_\varepsilon + \varepsilon\z)\|^2_{\H}+C \|\uprho_\varepsilon + \varepsilon\z+\u\|_{\widetilde{\L}^{r+1}}^{r-1} \|\varepsilon\z\|^2_{\widetilde{\L}^{r+1}}\nonumber\\&\quad+\frac{\beta}{4} \||\u|^{\frac{r-1}{2}}(\uprho_\varepsilon + \varepsilon\z)\|^2_{\H}+C \|\u\|_{\widetilde{\L}^{r+1}}^{r-1} \|\varepsilon\z\|^2_{\widetilde{\L}^{r+1}}.
		\end{align}
		By \eqref{MO_c}, we have
		\begin{align}\label{usc6}
		&-\beta\big\langle\mathcal{C}(\uprho_\varepsilon + \varepsilon\z+\u)-\mathcal{C}(\u),(\uprho_\varepsilon + \varepsilon\z+\u)-(\u)\big\rangle \nonumber\\&\leq-\frac{\beta}{2}\||\uprho_\varepsilon + \varepsilon\z+\u|^{\frac{r-1}{2}}((\uprho_\varepsilon + \varepsilon\z+\u)-(\u))\|_{\H}^2-\frac{\beta}{2}\||\u|^{\frac{r-1}{2}}((\uprho_\varepsilon + \varepsilon\z+\u)-(\u))\|_{\H}^2.
		\end{align}
		Using H\"older's and Young's inequalities, we get 
		\begin{align}
		\alpha|(\varepsilon\z, \uprho_\varepsilon)|& \leq \alpha \|\uprho_\varepsilon\|_{\H} \|\varepsilon\z\|_{\H} \leq \frac{1 }{2} \|\uprho_\varepsilon\|^2_{\H} + \frac{\alpha^2}{2} \|\varepsilon\z\|^2_{\H} \leq \frac{1}{2} \|\uprho_\varepsilon\|^2_{\H} + \frac{\alpha^2C}{2} \|\varepsilon\z\|^2_{\V}.\label{usc7}
		\end{align}
		Combining \eqref{usc2}-\eqref{usc7} and then substituting it in \eqref{usc1}, we find
		\begin{align}\label{usc8}
		&\frac{\d}{\d t} \|\uprho_\varepsilon(t)\|_{\H}^2 + \mu\|\uprho_\varepsilon(t)\|_{\V}^2 \nonumber\\&\leq C\bigg\{\|\uprho_\varepsilon(t) + \varepsilon\z(t)+\u(t)\|_{\widetilde{\L}^{r+1}}^{\frac{r+1}{r-1}}\|\uprho_\varepsilon(t) + \varepsilon\z(t)+\u(t)\|_{\V}^{\frac{r-3}{r-1}}+\|\u(t)\|_{\widetilde{\L}^{r+1}}^{\frac{r+1}{r-1}}\|\u(t)\|_{\V}^{\frac{r-3}{r-1}}\bigg\}\nonumber\\&\qquad\times\|\varepsilon\z(t)\|_{\V}+C\|\varepsilon\z(t)\|_{\V}^2+\wi\eta_4\|\uprho_\varepsilon(t)\|_{\H}^2+C\bigg\{\|\uprho_\varepsilon(t) + \varepsilon\z(t)+\u(t)\|_{\widetilde{\L}^{r+1}}^{r-1} +\|\u(t)\|_{\widetilde{\L}^{r+1}}^{r-1}\bigg\}\nonumber\\&\qquad\times\|\varepsilon\z(t)\|^2_{\widetilde{\L}^{r+1}} \nonumber\\&\leq C\bigg\{\|\u_\varepsilon(t)\|_{\widetilde{\L}^{r+1}}^{\frac{r+1}{r-1}}\|\u_\varepsilon(t)\|_{\V}^{\frac{r-3}{r-1}}+\|\u(t)\|_{\widetilde{\L}^{r+1}}^{\frac{r+1}{r-1}}\|\u(t)\|_{\V}^{\frac{r-3}{r-1}}\bigg\}\|\varepsilon\z(t)\|_{\V}+C\|\varepsilon\z(t)\|_{\V}^2+\wi\eta_4\|\uprho_\varepsilon(t)\|_{\H}^2\nonumber\\&\quad+C\bigg\{\|\u_\varepsilon(t)\|_{\widetilde{\L}^{r+1}}^{r-1} +\|\u(t)\|_{\widetilde{\L}^{r+1}}^{r-1}\bigg\}\|\varepsilon\z(t)\|^2_{\widetilde{\L}^{r+1}} 
		\end{align}
		where $\wi\eta_4=4\eta_4+1,$ for a.e. $t\in[-t_0,T]$. By integrating the above inequality from $-t_0$ to $t$, $t\in[-t_0,T]$, we get 
		\begin{align}\label{usc9}
		\|\uprho_\varepsilon(t)\|^2_{\H} + \mu \int_{-t_0}^{t} \|\uprho_\varepsilon(s) \|^2_{\V}\d s\leq \|\uprho_\varepsilon(-t_0)\|^2_{\H} + \int_{-t_0}^{t} \wi\eta_4 \|\uprho_\varepsilon(s)\|^2_{\H} \d s + \int_{-t_0}^{t} \upalpha_\varepsilon(s)  \d s, 
		\end{align}
		where
		\begin{align*}
		\upalpha_\varepsilon& = C\bigg\{\|\u_{\varepsilon}\|_{\widetilde{\L}^{r+1}}^{\frac{r+1}{r-1}}\|\u_{\varepsilon}\|_{\V}^{\frac{r-3}{r-1}}+\|\u\|_{\widetilde{\L}^{r+1}}^{\frac{r+1}{r-1}}\|\u\|_{\V}^{\frac{r-3}{r-1}}\bigg\}\|\varepsilon\z\|_{\V}+C\bigg\{\|\u_{\varepsilon}\|_{\widetilde{\L}^{r+1}}^{r-1}+\|\u\|_{\widetilde{\L}^{r+1}}^{r-1}\bigg\} \|\varepsilon\z\|^2_{\widetilde{\L}^{r+1}}\nonumber\\&\qquad+C \|\varepsilon\z\|^2_{\V}.
		\end{align*}
		Then by the classical Gronwall inequality, we deduce that 
		\begin{align}\label{usc10}
		\|\uprho_\varepsilon(t)\|^2_{\H}
		\leq\bigg(\frac{\varepsilon^2}{\lambda_1}\|\z(-t_0)\|^2_{\V}+\int_{-t_0}^{t}\upalpha_{\varepsilon}(s)\d s\bigg)e^{\wi\eta_4 (t+t_0)},
		\end{align}
		for all $t\in [-t_0,T]$. On the other hand, we have
		\begin{align*}
		&\int_{-t_0}^{t}\upalpha_\varepsilon(s)\d s\\&\leq\int_{-t_0}^{t} \bigg[\varepsilon C\bigg\{\|\u_{\varepsilon}(s)\|_{\widetilde{\L}^{r+1}}^{\frac{r+1}{r-1}}\|\u_{\varepsilon}(s)\|_{\V}^{\frac{r-3}{r-1}}+\|\u(s)\|_{\widetilde{\L}^{r+1}}^{\frac{r+1}{r-1}}\|\u(s)\|_{\V}^{\frac{r-3}{r-1}}\bigg\}\|\z(s)\|_{\V}+\varepsilon^2C\bigg\{\|(\u_{\varepsilon}(s))\|_{\widetilde{\L}^{r+1}}^{r-1}\\&\quad+\|(\u(s))\|_{\widetilde{\L}^{r+1}}^{r-1} \bigg\}\|\A\z(s)\|^2_{\H}+\varepsilon^2C \|\z(s)\|^2_{\V}\bigg]\d s,\\
		&	\leq\varepsilon C(t+t_0)^{1/2}\bigg[\|\u_{\varepsilon}\|^{\frac{r+1}{r-1}}_{\mathrm{L}^{r+1}([-t_0,t];\widetilde{\L}^{r+1})}\|\u_{\varepsilon}\|^{\frac{r-3}{r-1}}_{\mathrm{L}^{2}([-t_0,t];\V)} +\|\u\|^{\frac{r+1}{r-1}}_{\mathrm{L}^{r+1}([-t_0,t];\widetilde{\L}^{r+1})}\|\u\|^{\frac{r-3}{r-1}}_{\mathrm{L}^{2}([-t_0,t];\V)}\bigg]\\&\qquad\times\|\z\|_{\mathrm{L}^{\infty}([-t_0,t];\V)}+\varepsilon^2C\bigg[\|\u_{\varepsilon}\|^{r-1}_{\mathrm{L}^{r+1}([-t_0,t];\widetilde{\L}^{r+1})}+\|\u\|^{r-1}_{\mathrm{L}^{r+1}([-t_0,t];\widetilde{\L}^{r+1})}\bigg]\|\z\|^2_{\mathrm{L}^{r+1}([-t_0,t];\D(\A))}\\&\quad+\varepsilon^2C(t+t_0)\|\z\|^2_{\mathrm{L}^{\infty}([-t_0,t];\V)}.
		\end{align*}
		Since the processes $\u_\varepsilon, \u\in \mathrm{L}^{\infty}_{\mathrm{loc}}([-t_0,\infty);\H)\cap\mathrm{L}^{2}_{\mathrm{loc}}([-t_0,\infty);\V)\cap\mathrm{L}^{r+1}_{\mathrm{loc}}([-t_0,\infty);\widetilde{\L}^{r+1})$, and $\z\in\mathrm{L}^{\infty}_{\mathrm{loc}}([-t_0,\infty);\V)\cap\mathrm{L}^{r+1}_{\mathrm{loc}}([-t_0,\infty);\D(\A))$,  therefore  $\int_{-t_0}^{t}\upalpha_\varepsilon(s)\d s\to 0 \text{ as } \varepsilon\to 0$. Hence, by \eqref{usc10}, we immediately have  $$\lim_{\varepsilon\to 0^+}\|\uprho_\varepsilon(t)\|^2_{\H}=0,$$ which completes the proof of \eqref{usc} by taking $t=0$. Hence $(K_2)$ is verified for $r>3$.\vskip 0.2 cm\noindent 
		\textbf{Case II:} $r=3$ and $2\beta\mu\geq1$.
		By H\"older's inequality, \eqref{lady} and \eqref{poin}, we have
		\begin{align}
		|\big\langle\B(\uprho_\varepsilon + \varepsilon\z+\u),\varepsilon \z\big\rangle|&\leq\|\uprho_\varepsilon + \varepsilon\z+\u\|_{\widetilde{\L}^{4}}^{2}\|\varepsilon\z\|_{\V}\leq C \|\uprho_\varepsilon + \varepsilon\z+\u\|_{\V}^{2}\|\varepsilon\z\|_{\V},\label{usc11}\\
		|\big\langle\B(\u), \varepsilon\z\big\rangle|&\leq\|\u\|_{\widetilde{\L}^{4}}^{2}\|\varepsilon\z\|_{\V}\leq C\|\u\|_{\V}^{2}\|\varepsilon\z\|_{\V}.\label{usc12}
		\end{align}
		Using \eqref{441} and H\"older's inequality, we find 
		\begin{align}\label{usc13}
		&|\langle\B(\uprho_\varepsilon + \varepsilon\z+\u)-\B(\u),(\uprho_\varepsilon + \varepsilon\z+\u)-(\u)\rangle|\nonumber\\&\leq\|\uprho_\varepsilon + \varepsilon\z\|_{\V}\||\u|(\uprho_\varepsilon + \varepsilon\z)\|_{\H}\leq\frac{1}{2\beta}(\|\uprho_\varepsilon\|_{\V}+\|\varepsilon\z\|_{\V})^2+\frac{\beta}{2}\||\u|(\uprho_\varepsilon + \varepsilon\z)\|_{\H}^2\nonumber\\&\leq\frac{1}{2\beta}\|\uprho_\varepsilon\|_{\V}^2+ \frac{3}{2\beta}\|\varepsilon\z\|_{\V}^2+ \frac{1}{\beta}\|\u_{\varepsilon}\|_{\V}\|\varepsilon\z\|_{\V}+\frac{1}{\beta}\|\u\|_{\V}\|\varepsilon\z\|_{\V}+\frac{\beta}{2}\||\u|(\uprho_\varepsilon + \varepsilon\z)\|_{\H}^2.
		\end{align}
		Making use of H\"older's inequality, \eqref{lady} and \eqref{poin}, we obtain 
		\begin{align}\label{usc14}
		\big\langle\mathcal{C}(\uprho_\varepsilon + \varepsilon\z+\u)-\mathcal{C}(\u),\varepsilon \z\big\rangle&\leq\big\{\|\uprho_\varepsilon + \varepsilon\z+\u\|^3_{\L^4}+\|\u\|^3_{\L^4}\big\}\|\varepsilon\z\|_{\L^4}\nonumber\\&\leq C\big\{\|\u_{\varepsilon}\|^3_{\L^4}+\|\u\|^3_{\L^4}\big\}\|\varepsilon\z\|_{\V}.
		\end{align}
		By \eqref{MO_c}, we have
		\begin{align}\label{usc15}
		-\beta\big\langle\mathcal{C}(\uprho_\varepsilon + \varepsilon\z+\u)-\mathcal{C}(\u),(\uprho_\varepsilon + \varepsilon\z+\u)-(\u)\big\rangle\leq-\frac{\beta}{2}\||\u|(\uprho_\varepsilon + \varepsilon\z)\|_{\H}^2.
		\end{align}
		Using \eqref{usc11}-\eqref{usc15} with \eqref{usc7} in \eqref{usc1}, we arrive at
		\begin{align}\label{usc16}
		&\quad\frac{\d}{\d t} \|\uprho_\varepsilon(t)\|_{\H}^2 + 2\bigg(\mu-\frac{1}{2\beta}\bigg)\|\uprho_\varepsilon(t)\|_{\V}^2 \nonumber\\&\leq C\bigg\{ \|\u_{\varepsilon}(t)\|_{\V}^{2} +\|\u(t)\|_{\V}^{2}+ \|\u_{\varepsilon}(t)\|_{\V}+\|\u(t)\|_{\V}+\|\u_{\varepsilon}(t)\|^3_{\L^4}+\|\u(t)\|^3_{\L^4}\bigg\}\|\varepsilon \z(t)\|_{\V}\nonumber\\&\quad+\|\uprho_\varepsilon(t)\|^2_{\H}+C \|\varepsilon\z(t)\|^2_{\V},
		\end{align}
		for a.e. $t\in[-t_0,T]$. By integrating the above inequality from $-t_0$ to $t$, $t\in[-t_0,T]$, we get 
		\begin{align}\label{usc17}
		\|\uprho_\varepsilon(t)\|^2_{\H} \leq \|\uprho_\varepsilon(-t_0)\|^2_{\H} +2 \int_{-t_0}^{t}  \|\uprho_\varepsilon(s)\|^2_{\H} \d s + \int_{-t_0}^{t} \upbeta_\varepsilon(s)  \d s, 
		\end{align}
		where
		\begin{align*}
		\upbeta_\varepsilon =&C\bigg\{ \|\u_{\varepsilon}\|_{\V}^{2} +\|\u\|_{\V}^{2}+ \|\u_{\varepsilon}\|_{\V}+\|\u\|_{\V}+\|\u_{\varepsilon}\|^3_{\L^4}+\|\u\|^3_{\L^4}\bigg\}\|\varepsilon\z\|_{\V}+C \|\varepsilon\z\|^2_{\V}.
		\end{align*}
		Then by the classical Gronwall inequality, we deduce that 
		\begin{align}\label{usc18}
		\|\uprho_\varepsilon(t)\|^2_{\H}\leq\bigg(\varepsilon^2\|\z(-t_0)\|^2_{\H}+\int_{-t_0}^{t}\upbeta_{\varepsilon}(s)\d s\bigg)e^{2 (t+t_0)}
		\leq\bigg(\frac{\varepsilon^2}{\lambda_1}\|\z(-t_0)\|^2_{\V}+\int_{-t_0}^{t}\upbeta_{\varepsilon}(s)\d s\bigg)e^{2 (t+t_0)}.
		\end{align}
		for all $t\in [-t_0,T]$. On the other hand, we have	\begin{align*}
		&\quad\int_{-t_0}^{t}\upbeta_{\varepsilon}(s)\d s\nonumber\\ &
		\leq\varepsilon C\bigg\{\|\u_{\varepsilon}\|_{\mathrm{L}^{2}([-t_0,t];\V)}+\|\u\|_{\mathrm{L}^{2}([-t_0,t];\V)} +(t+t_0)^{1/2}\|\u_{\varepsilon}\|_{\mathrm{L}^{2}([-t_0,t];\V)}+(t+t_0)^{1/2}\|\u\|_{\mathrm{L}^{2}([-t_0,t];\V)}\nonumber\\&\quad\quad\quad+(t+t_0)^{1/4}\|\u_{\varepsilon}\|^3_{\mathrm{L}^4([-t_0,t];\widetilde{\L}^4)}+(t+t_0)^{1/4}\|\u\|^3_{\mathrm{L}^4([-t_0,t];\widetilde{\L}^4)}\bigg\}\|\z\|_{\mathrm{L}^{\infty}([-t_0,t];\V)}\nonumber\\&\quad\quad\quad+\varepsilon^2(t+t_0)\|\z\|^2_{\mathrm{L}^{\infty}([-t_0,t];\V)}.
		\end{align*}
		Since the processes $\u_\varepsilon, \u\in \mathrm{L}^{\infty}_{\mathrm{loc}}([-t_0,\infty);\H)\cap\mathrm{L}^{2}_{\mathrm{loc}}([-t_0,\infty);\V)\cap\mathrm{L}^{4}_{\mathrm{loc}}([-t_0,\infty);\widetilde{\L}^{4})$, and $\z\in\mathrm{L}^{\infty}_{\mathrm{loc}}([-t_0,\infty);\V)\cap\mathrm{L}^{4}_{\mathrm{loc}}([-t_0,\infty);\D(\A))$,  therefore  $\int_{-t_0}^{t}\upbeta_\varepsilon(s)\d s\to 0 \text{ as } \varepsilon\to 0$. Hence, by \eqref{usc18}, we immediately have  $$\lim_{\varepsilon\to 0^+}\|\uprho_\varepsilon(t)\|^2_{\H}=0,$$ which completes the proof of \eqref{usc} by taking $t=0$. Hence $(K_2)$ is verified for $r=3$.
		
		Since both the conditions $(K_1)$ and $(K_2)$  hold for our model, the property of upper semi-continuity \eqref{51} holds true in $\H$.
		\vskip 2mm
		\noindent
		\textbf{Step III.} \emph{Proof of \eqref{K_3}:} In order to prove \eqref{K_3}, it is enough to show that for any bounded subset $\G\subset \H,$ we have 
		\begin{align}\label{usc81}
		\lim_{\varepsilon\to\varepsilon_0}\sup_{\x\in \G} \|\varphi_{\varepsilon}(t_0, \theta_{-t_0}\omega)\x-\varphi_{\varepsilon_0}(t_0, \theta_{-t_0}\omega)\x\|_{\H}=0.
		\end{align}
		For any $\x\in \G,$ let us take $\u_\varepsilon(t)=\varphi_{\varepsilon}(t+t_0,\theta_{-t_0})\x$ and $\u_{\varepsilon_0}(t)=\varphi_{\varepsilon_0}(t+t_0,\theta_{-t_0})\x$. Let $\u_{\varepsilon}(\cdot)$ be the unique weak solution of the system \eqref{S-CBF} and $\u_{\varepsilon_0}(\cdot)$ be the unique weak solution of the system \eqref{S-CBF} when $\varepsilon$ replaced by $\varepsilon_0$, with initial condition $\x$ at $t=-t_0$. Also, let $$\y(t)= \u_\varepsilon(t)-\u_{\varepsilon_0}(t), \ \ t\in[0,T].$$ 	Clearly $\y(\cdot)$ satisfies
		\begin{equation}\label{cSCBF_w}
		\left\{
		\begin{aligned}
		\d\y+\{\mu \A\y+\B(\y+\u_{\varepsilon_0})-\B(\u_{\varepsilon_0})+\beta \mathcal{C}(\y+\u_{\varepsilon_0})-\beta\mathcal{C}(\u_{\varepsilon_0})\}\d t&= \varepsilon^*\d\text{W}(t),  \\ 
		\y(-t_0)&=\mathbf{0},
		\end{aligned}
		\right.
		\end{equation}
		in $\V'+\wi\L^{\frac{r+1}{r}}$ for all $t\in[-t_0,T]$, where $\e^*=\e-\e_0$. 	Let us introduce $\w(\cdot)=\y(\cdot)-\varepsilon^*\z(\cdot)$, where $\z(\cdot)$ is the unique solution of \eqref{OUPe}. Then $\w(\cdot)$ satisfies the following equation in $\V'+\wi\L^{\frac{r+1}{r}}$:
		\begin{equation}\label{cscbf_rho2}
		\left\{
		\begin{aligned}
		\frac{\d\w}{\d t} &= -\mu \A\w - \B(\w + \varepsilon^*\z+\u)+\B(\u) - \beta \mathcal{C}(\w + \varepsilon^*\z +\u)+\beta\mathcal{C}(\u) +\varepsilon^* \alpha \z, \\
		\w(-t_0)&= -\varepsilon^* \z(-t_0).
		\end{aligned}
		\right.
		\end{equation}
		The above system is similar to \eqref{cscbf_eta1} and a calculation similar to \eqref{usc9} (for $r>3$) and \eqref{usc17} (for $r=3$ and $2\beta\mu\geq1$) yields \eqref{52}. 
	\end{proof}

	\medskip\noindent
	{\bf Acknowledgments:}    The first author would like to thank the Council of Scientific $\&$ Industrial Research (CSIR), India for financial assistance (File No. 09/143(0938)/2019-EMR-I).  M. T. Mohan would  like to thank the Department of Science and Technology (DST), Govt of India for Innovation in Science Pursuit for Inspired Research (INSPIRE) Faculty Award (IFA17-MA110).

\end{document}